\documentclass{amsart}
\usepackage{amsfonts,amssymb,verbatim,amsmath,amsthm,latexsym,textcomp,amscd}
\usepackage[dvipsnames]{xcolor}
\usepackage{tikz,tikz-cd}
\usepackage{graphics,textcomp,hyperref}
\usepackage{graphicx}
\usepackage{url}
\usepackage{psfrag}
\usepackage{stackrel}
\usepackage{subfiles} 
\usepackage{todonotes}
\usepackage{enumitem}
\usepackage{mathtools}
\usepackage[capitalise]{cleveref}

\DeclareMathOperator{\Stab}{Stab}

\newtheorem{theorem}{Theorem}[section]
\newtheorem{prop}[theorem]{Proposition}
\newtheorem{lemma}[theorem]{Lemma}

\newtheorem{cor}[theorem]{Corollary}

\newtheorem{claim}[theorem]{Claim}
\newtheorem{qn}[theorem]{Question}

\newtheorem{theoremA}{Theorem}

\newtheoremstyle{named}{}{}{\itshape}{}{\bfseries}{.}{.5em}{\thmnote{#3}}
\theoremstyle{named}
\newtheorem*{namedtheorem}{Theorem}

\theoremstyle{definition}
\newtheorem{defn}[theorem]{Definition}
\newtheorem{rmk}[theorem]{Remark}

\newcommand{\hhat}{\widehat}

\newcommand{\natls}{{\mathbb N}}

\newcommand\AAA{{\mathcal A}}

\newcommand\BB{{\mathcal B}}
\newcommand\CC{{\mathcal C}}
\newcommand\DD{{\mathcal D}}
\newcommand\EE{{\mathcal E}}
\newcommand\FF{{\mathcal F}}
\newcommand\GG{{\mathcal G}}
\newcommand\HH{{\mathcal H}}

\newcommand\KK{{\mathcal K}}
\newcommand\LL{{\mathcal L}}
\newcommand\MM{{\mathcal M}}

\newcommand\PP{{\mathcal P}}

\newcommand\SSS{{\mathcal S}}

\newcommand\UU{{\mathcal U}}
\newcommand\VV{{\mathcal V}}

\newcommand\YY{{\mathcal Y}}
\newcommand\ZZ{{\mathcal Z}}

\newcommand\PMF{{\PP\kern-2pt\MM\FF}}

\newcommand\PML{{\PP\kern-2pt\MM\LL}}

\newcommand\ep{\epsilon}

\newcommand\Hyp{{\mathbb H}}

\newcommand\R{{\mathbb R}}

\newcommand{\bfC}{{\mathbf{C}}}

\newcommand{\diam}{\operatorname{diam}}

\newcommand{\fsubd}{\mathrel{{\scriptstyle\searrow}\kern-1ex^d\kern0.5ex}}
\newcommand{\bsubd}{\mathrel{{\scriptstyle\swarrow}\kern-1.6ex^d\kern0.8ex}}
\newcommand{\fsubeq}{\mathrel{\raise-.7ex\hbox{$\overset{\searrow}{=}$}}}
\newcommand{\bsubeq}{\mathrel{\raise-.7ex\hbox{$\overset{\swarrow}{=}$}}}

\renewcommand{\dh}{d_{\rm{Haus}}}

\makeatletter
\DeclareFontFamily{U}{tipa}{}
\DeclareFontShape{U}{tipa}{m}{n}{<->tipa10}{}
\newcommand{\arc@char}{{\usefont{U}{tipa}{m}{n}\symbol{62}}}%

\newcommand{\arc}[1]{\mathpalette\arc@arc{#1}}

\newcommand{\arc@arc}[2]{%
	\sbox0{$\m@th#1#2$}%
	\vbox{
		\hbox{\resizebox{\wd0}{\height}{\arc@char}}
		\nointerlineskip
		\box0
	}%
}
\makeatother

\newcommand{\set}[1]{\left\{ #1 \right\}}
\DeclareMathOperator{\Hdim}{Hdim}

\newcommand{\bbH}{\mathbb{H}}

\newcommand{\bbN}{\mathbb{N}}

\newcommand{\bbR}{\mathbb{R}}

\newcommand{\bbZ}{\mathbb{Z}}
\newcommand{\calS}{\mathcal{S}}
\newcommand{\calA}{\mathcal{A}}

\DeclareMathOperator{\bdr}{bdr}
\DeclareMathOperator{\intr}{int}
\DeclareMathOperator{\cl}{cl}
\DeclareMathOperator{\hull}{Hull}
\DeclareMathOperator{\Isom}{Isom}
\DeclareMathOperator{\Sh}{Sh}
\DeclareMathOperator{\br}{bar}
\DeclareMathOperator{\stab}{stab}
\DeclareMathOperator{\Aut}{Aut}
\DeclareMathOperator{\Comm}{Comm}
\newcommand{\whull}{\hull_w}
\RequirePackage[normalem]{ulem} 
\RequirePackage{color}\definecolor{RED}{rgb}{1,0,0}\definecolor{BLUE}{rgb}{0,0,1} 
\providecommand{\DIFdeltex}[1]{}                      
\providecommand{\DIFaddbegin}{} 
\providecommand{\DIFaddend}{} 
\providecommand{\DIFdelbegin}{} 
\providecommand{\DIFdelend}{} 
\providecommand{\DIFaddbeginFL}{} 
\providecommand{\DIFaddendFL}{} 
\providecommand{\DIFdelbeginFL}{} 
\providecommand{\DIFdelendFL}{} 
\newcommand{\DIFscaledelfig}{0.5}
\RequirePackage{settobox} 
\RequirePackage{letltxmacro} 
\newsavebox{\DIFdelgraphicsbox} 
\newlength{\DIFdelgraphicswidth} 
\newlength{\DIFdelgraphicsheight} 
\LetLtxMacro{\DIFOincludegraphics}{\includegraphics} 
\newcommand{\DIFaddincludegraphics}[2][]{{\color{blue}\fbox{\DIFOincludegraphics[#1]{#2}}}} 
\newcommand{\DIFdelincludegraphics}[2][]{
\sbox{\DIFdelgraphicsbox}{\DIFOincludegraphics[#1]{#2}}
\settoboxwidth{\DIFdelgraphicswidth}{\DIFdelgraphicsbox} 
\settoboxtotalheight{\DIFdelgraphicsheight}{\DIFdelgraphicsbox} 
\scalebox{\DIFscaledelfig}{
\parbox[b]{\DIFdelgraphicswidth}{\usebox{\DIFdelgraphicsbox}\\[-\baselineskip] \rule{\DIFdelgraphicswidth}{0em}}\llap{\resizebox{\DIFdelgraphicswidth}{\DIFdelgraphicsheight}{
\setlength{\unitlength}{\DIFdelgraphicswidth}
\begin{picture}(1,1)
\thicklines\linethickness{2pt} 
{\color[rgb]{1,0,0}\put(0,0){\framebox(1,1){}}}
{\color[rgb]{1,0,0}\put(0,0){\line( 1,1){1}}}
{\color[rgb]{1,0,0}\put(0,1){\line(1,-1){1}}}
\end{picture}
}\hspace*{3pt}}} 
} 
\LetLtxMacro{\DIFOaddbegin}{\DIFaddbegin} 
\LetLtxMacro{\DIFOaddend}{\DIFaddend} 
\LetLtxMacro{\DIFOdelbegin}{\DIFdelbegin} 
\LetLtxMacro{\DIFOdelend}{\DIFdelend} 
\DeclareRobustCommand{\DIFaddbegin}{\DIFOaddbegin \let\includegraphics\DIFaddincludegraphics} 
\DeclareRobustCommand{\DIFaddend}{\DIFOaddend \let\includegraphics\DIFOincludegraphics} 
\DeclareRobustCommand{\DIFdelbegin}{\DIFOdelbegin \let\includegraphics\DIFdelincludegraphics} 
\DeclareRobustCommand{\DIFdelend}{\DIFOaddend \let\includegraphics\DIFOincludegraphics} 
\LetLtxMacro{\DIFOaddbeginFL}{\DIFaddbeginFL} 
\LetLtxMacro{\DIFOaddendFL}{\DIFaddendFL} 
\LetLtxMacro{\DIFOdelbeginFL}{\DIFdelbeginFL} 
\LetLtxMacro{\DIFOdelendFL}{\DIFdelendFL} 
\DeclareRobustCommand{\DIFaddbeginFL}{\DIFOaddbeginFL \let\includegraphics\DIFaddincludegraphics} 
\DeclareRobustCommand{\DIFaddendFL}{\DIFOaddendFL \let\includegraphics\DIFOincludegraphics} 
\DeclareRobustCommand{\DIFdelbeginFL}{\DIFOdelbeginFL \let\includegraphics\DIFdelincludegraphics} 
\DeclareRobustCommand{\DIFdelendFL}{\DIFOaddendFL \let\includegraphics\DIFOincludegraphics} 
\RequirePackage{listings} 
\RequirePackage{color} 
\lstdefinelanguage{DIFcode}{ 
  moredelim=[il][\color{red}\sout]{\%DIF\ <\ }, 
  moredelim=[il][\color{blue}\uwave]{\%DIF\ >\ } 
} 
\lstdefinestyle{DIFverbatimstyle}{ 
	language=DIFcode, 
	basicstyle=\ttfamily, 
	columns=fullflexible, 
	keepspaces=true 
} 
\lstnewenvironment{DIFverbatim}{\lstset{style=DIFverbatimstyle}}{} 
\lstnewenvironment{DIFverbatim*}{\lstset{style=DIFverbatimstyle,showspaces=true}}{} 

\begin{document}

\title{Commensurated hyperbolic subgroups}

\author{Nir Lazarovich}

\address{Department of Mathematics
	Technion
	Haifa 32000
	Israel}

\email{lazarovich@technion.ac.il}

\author{Alex Margolis}

\address{Alex Margolis, Department of Mathematics, The Ohio State University,  Mathematics Tower,  231 W 18th Ave,  Columbus,  OH  43210, USA}
\email{margolis.93@osu.edu}
\urladdr{https://sites.google.com/view/alexmargolis/}

\author{Mahan Mj}

\address{School of Mathematics, Tata Institute of Fundamental Research, Mumbai-40005, India}
\email{mahan@math.tifr.res.in}
\email{mahan.mj@gmail.com}
\urladdr{http://www.math.tifr.res.in/~mahan}


\thanks{AM and MM were supported in part by  the Institut Henri Poincare  (UAR 839 CNRS-Sorbonne Universite), LabEx CARMIN, ANR-10-LABX-59-01, during their participation in the trimester program "Groups acting on fractals, Hyperbolicity and Self-similarity", April-June 2022. NL and MM thank ICTS Bengaluru for its hospitality during a program on "Probabilistic Methods in Negative Curvature" in February 2023. NL was partially supported by the Israeli Science Foundation (grant no. 1576/23).
		MM is supported by  the Department of Atomic Energy, Government of India, under project no.12-R\&D-TFR-5.01-0500,  by an endowment of the Infosys Foundation.
	and by   a DST JC Bose Fellowship. }

\begin{abstract}
We show that if $H$ is a non-elementary hyperbolic commensurated subgroup of infinite index in a hyperbolic group $G$, then $H$  is virtually a free product of hyperbolic surface groups and free groups. 
We prove that whenever a one-ended hyperbolic group $H$ is a fiber of a non-trivial hyperbolic bundle then $H$ virtually splits over a 2-ended subgroup.
\end{abstract}

\maketitle

\date{\today}

\section{Introduction}\label{sec-intro}
Our starting point in this paper is the following deep theorem that is a consequence of  work of Bestvina, Paulin, Rips and Sela; see \cite[p. 379]{mitra-endlam} for an explanation.

\begin{theorem}\label{intro-hypnormal}
Let $H$ be a hyperbolic normal subgroup of infinite index in a hyperbolic group $G$. Then $H$ is virtually a free product of hyperbolic surface groups and free groups. 
\end{theorem}

A subgroup $H\le G$ is \emph{commensurated} (sometimes called \emph{almost-normal}) if for all $g\in G$, $gHg^{-1}$ is commensurable to $H$, i.e. their intersection $H\cap gHg^{-1}$ has finite index in both $H$ and $gHg^{-1}$.
In this paper, we extend Theorem~\ref{intro-hypnormal} to commensurated subgroups.

\begin{theoremA}\label{intro-hypcommens}
	Let $H$ be a hyperbolic commensurated subgroup of infinite index in a hyperbolic group $G$. Then $H$  is virtually a free product of hyperbolic surface groups and free groups. 
\end{theoremA}

A group $G$ as in the theorem gives rise to a metric graph bundle (in the sense of \cite{mahan-sardar} or a  coarse bundle in the sense of \cite{margolis-gt}) whose fibers are quasi-isometric to $H$. We refer the reader to \cref{sec-mbdl-dendrite} for definitions (especially \cref{defn-mgbdl}).
The key ingredient of the proof of \cref{intro-hypcommens} is the following more geometric theorem about such bundles.

\begin{theoremA}\label{intro-main}
Let $H$ be a one-ended hyperbolic group. If $H$ is quasi-isometric to a fiber of a hyperbolic metric graph bundle $p: X \to B$ with controlled hyperbolic fibers, bounded valence and unbounded base, then $H$ virtually splits over a 2-ended subgroup.
\end{theoremA}

\subsection{About the proofs.} 

\medskip

A comment is in order pertaining to the similarities and differences in the techniques used in
the proofs of \cref{intro-hypnormal,intro-hypcommens}.
Using Stallings' theorem on ends of groups \cite{stallings1968torsion}, and 
Dunwoody's accessibility result \cite{dunwoody1985accessibility},  both proofs  start with the maximal graph of groups decomposition with finite edge groups.
Uniqueness of this decomposition reduces the problem to the case where each
factor $H_v$  of this  Dunwoody-Stallings decomposition is one-ended.
The next step in both proofs is to use the canonical JSJ decomposition of Rips-Sela-Bowditch \cite{rips-sela,Sela,rips-sela-jsj,bowditch-cutpts}. The output of the JSJ decomposition when
$H_v$ is \emph{not virtually a surface group}  is a graph of groups, where 
\begin{enumerate}
\item each edge group is 2-ended
\item each vertex group is either rigid, or maximal  hanging Fuchsian, or 2-ended.
\end{enumerate}
By a theorem of Bowditch \cite{bowditch-cutpts} refining the decomposition of Rips-Sela \cite{rips-sela,Sela,rips-sela-jsj}, the above decomposition is unique, and can, in fact be read 
off topologically from the boundary $\partial H_v$. Since $\partial H_v$ is invariant under quasi-isometries, automorphisms preserve the above JSJ decomposition. In a slightly
subtler sense, so do commensurations. This forces the above JSJ decomposition to be trivial, i.e.\ it has no edge groups.
Thus, to conclude that $H_v$ is virtually a surface group it remains only to show that it splits over a 2-ended subgroup. At this stage, the proofs of \cref{intro-hypnormal,intro-hypcommens}
diverge.

 In the proof of \cref{intro-hypnormal}, one takes a rescaling limit \`a la Bestvina \cite{bestvina-duke} and Paulin \cite{paulin-inv} to obtain a non-trivial action of $H_v$ on an $\R-$tree. Then, using the Rips machine \cite{BF-rtrees}, one concludes that $H_v$ splits over a 2-ended subgroup. 
 In the proof of \cref{intro-hypcommens}, a similar strategy would fail, as the limiting $\R$-tree does not usually carry an action of $H_v$. 
 Instead, we adopt a coarse bundle perspective in the spirit of \cite{mahan-sardar,margolis-gt}. Since $H$ is commensurated in $G$, we can view $G$ as a bundle whose fibers are isometric to $H$. From this bundle we obtain a hyperbolic bundle $p:X\to [0,\infty)$ whose fibers are quasi-isometric to $H_v$. The desired splitting of $H_v$ will now follow from \cref{intro-main}.

 We sketch here a proof of \cref{intro-main}: Assume $H$ and $p:X\to B$ are as in \cref{intro-main}.
By Bowditch \cite{bowditch-stacks}, $\partial X$ is a dendrite. By \cite{mahan-sardar}, there is a Cannon-Thurston map $\partial i: \partial H \to \partial X$. The preimage $\ZZ=\partial i^{-1}(\xi)$ of a separating point $\xi$ in the dendrite $\partial X$ is a separating subset of $\partial H$. \cref{lem:polynomial growth of horospheres} shows that its weak convex hull $\CC_0$ has polynomial growth in $H$. \cref{pos hdim implies exponential growth} implies that the Hausdorff dimension of $\ZZ$ is 0. Finally, \cref{prop: Hdim of qc separator} shows if $H$ does not virtually split over a 2-ended subgroup then the Hausdorff dimension of any separating subset is positive. Hence $H$ must virtually split over a 2-ended subgroup. We state 
\cref{prop: Hdim of qc separator} explicitly here as it might be of independent interest:

 \begin{namedtheorem}[\cref{prop: Hdim of qc separator}]\label{intro-growth}
		Let $H$ be a one-ended hyperbolic group which does not virtually split over a 2-ended subgroup. For a visual metric on $\partial H$, there exists $\epsilon>0$ such that if $\ZZ\subseteq \partial H$ is separating then $\ZZ$ has Hausdorff dimension 
		at least $\ep$ with respect to this visual metric.
	 \end{namedtheorem}

\subsection{Examples of commensurated subgroups} We conclude this introduction by recalling examples of hyperbolic commensurated subgroups of hyperbolic groups.
In \cite[Theorem 1.1]{min},
Min  constructs  a graph of groups where
\begin{enumerate}
	\item Each edge and vertex group is a closed hyperbolic surface 
	group.
	\item The inclusion map of an edge group
	into a vertex group takes the edge group injectively onto a subgroup
	of finite index in the vertex group.
	\item The resulting graph of groups is hyperbolic.
\end{enumerate}

By choosing the maps in the above construction carefully,
Min furnishes examples \cite[Section 5]{min} that are
not abstractly commensurable to a surface-by-free group.  However, since the graphs of hyperbolic groups
in this construction satisfy   the qi-embedded condition, they give rise to metric graph bundles. 

Min's examples from \cite{min} were generalized 
to the context of free groups in \cite{ghosh-mj}, where all the edge and vertex groups are finitely generated free.

What Theorem~\ref{intro-hypcommens} shows is that, as far as $H$ is concerned, these
two sets of examples, i.e.\ surface groups and free groups, are essentially all that can occur. However, a general understanding of what $G$ can be is missing.
A well-known open question asks if one can have an exact sequence of hyperbolic groups $1 \to H \to G \to Q \to 1$, where $H$ is a free group or a surface group,
and $Q$ is one-ended hyperbolic. 
The following question generalizes this to the context of commensurations. We refer to  \cref{sec-ca} for  definitions of the Cayley--Abels graph and Schlichting completion  of the pair $(G,H)$.

\begin{qn}\label{qn-one-ended}
Let $H$ be a non-elementary hyperbolic commensurated subgroup of infinite
index in a hyperbolic group $G$, and let $\GG$ be a Cayley--Abels graph of $(G,H)$. Can 
$\GG$ be one-ended?
\end{qn}
We note that for $G$, $H$, $\GG$ as in \mbox{
\cref{qn-one-ended}}\hskip0pt
, \mbox{
\cref{cor-qhyp} }\hskip0pt
ensures that the Cayley--Abels graph $\GG$ is hyperbolic.
 We make some brief observations about the structure of $G$,  $H$ and $\GG$ as above (see Section~\ref{sec-prop-restrns}):
\begin{prop}\label{prop:conclusions}
	Let $G$, $H$, $\GG$ be as in Question \ref{qn-one-ended}.
	\begin{enumerate}
		\item If $\GG$ is quasi-isometric to a rank-one symmetric space, then $H$ contains a finite index subgroup that is normal in $G$.
		\item If $\partial\GG$ is homeomorphic to $S^n$ for $0\leq n\leq 3$, then $H$ contains a finite index subgroup that is normal in $G$.
		\item If $|\partial \GG|>2$, the action of $G$ on $\GG$ does not fix a point of $\partial \GG$. In particular, the Schlichting completion of $(G,H)$ is non-amenable.
	\end{enumerate}
\end{prop}

\subsection{A historical aside}
In \cite{bowditch-stacks}, Bowditch, using an argument he attributes to Kleiner, shows that if $P: X \to \R$ is a metric graph bundle, then the boundary 
 $\partial F_0$ of the fiber $F_0$ over $\{0\}$  embeds in a product of dendrites, and in particular is at most 2-dimensional. 

 Next suppose that $F_0$ admits a  geometric $H-$action, so that $\partial H =
 \partial F_v$.
 It is not hard to show using work of Bestvina-Mess \cite{bes-mess} that, even if 
 $\partial H$ is 2-dimensional, it cannot embed in a 2-dimensional contractible space. Thus, in the setup of a   metric graph bundle $P: X \to \R$
 one is reduced to the case of considering hyperbolic groups $H$ with
 one-dimensional boundary.
Work of \cite{kapovich-kleiner} implies that if $H$ does not split over $\bbZ$,
it must be either a Sierpinski carpet or a Menger curve. Theorem~\ref{intro-main}
rules these cases out.

We further point out that in the setup of Theorem~\ref{intro-main}, we only have
one dendrite and not two as in \cite{bowditch-stacks}. Thus, the embedding result
above in a product of dendrites cannot be used.

\section{Preliminaries on Metric Graph Bundles and JSJ}\label{sec-mbdl-dendrite}

\subsection{Metric Graph Bundles}\label{sec-mbdl} For the purposes of this paper, a \emph{metric graph} will refer to a connected graph equipped with the natural metric assigning length one  to each edge. This is to maintain consistency with the terminology borrowed from \cite{mahan-sardar}. We say that an inclusion 
$i:(Y,d_Y) \to (X,d_X)$ of metric spaces is \emph{metrically proper as measured by a proper function $f: \natls \to \natls$} if for all $x, y \in Y$, $d_X(x,y) \leq n$ implies $d_Y(x,y) \leq f(n)$. 
For a metric graph $X$, the vertex set will be denoted as $\VV(X)$.
\begin{defn}\label{defn-mgbdl}\cite[Definition 1.2]{mahan-sardar}
	Suppose $X$ and $B$ are metric graphs.
	We say that a surjective simplicial 
	map $p:X\to B$ is a {\em metric graph bundle} if there exists a function $f:\bbN\to \bbN$  such that the following hold.
	\begin{enumerate}
		\item For all $b\in \VV(B)$, $F_b:=p^{-1}(b)$ is a connected subgraph of $X$. Moreover, the inclusion maps
		$F_b\rightarrow X$, $b\in \VV(B)$ are uniformly (independent of $b$) metrically proper as measured by $f$.
		\item For all adjacent vertices $b_1,b_2\in \VV(B)$, any $x_1\in \VV(F_{b_1})$ is connected 
		by an edge to some $x_2\in \VV(F_{b_2})$.
	\end{enumerate}
	We  say $X$ is a \emph{metric graph bundle over $B$} if there exists a map $p:X\to B$ as above.
For all $b\in \VV(B)$ we shall refer to $F_b$ as the {\em fiber} over $b$ and denote its (intrinsic) path metric by $d_b$.
\end{defn}

Letting $\dh$ denote Hausdorff distance, condition (2) of Definition \ref{defn-mgbdl} immediately gives:

\begin{lemma}\label{bundle lemma1}
	If $p:X\rightarrow B$ is a metric graph bundle, then for any points $v,w\in \VV(B)$ we have
	$\dh(F_v, F_w)<\infty$.
\end{lemma}

\subsection{Barycentric maps and controlled hyperbolic fibers}

We will be interested in metric graph bundles where the fibers are hyperbolic.
We assume that the reader is familiar with hyperbolic spaces and their boundaries (cf. \cite{gromov-hypgps},\cite[Chapter III.H]{bridson-haefliger} and \cite{ghys1990espaces}). 
We recall that for a hyperbolic space, one can define a barycenter using the following lemma.
\begin{lemma}\label{lem-bary}
	Suppose $X$ is proper $\delta$-hyperbolic. Then:\\
	(1) {\em (Visibility \cite[Chapter III.H, Lemma 3.2]{bridson-haefliger}).}  Suppose $\xi_1, \xi_2\in \partial X$ are two distinct points. Then there is a geodesic line in $X$ joining them.\\
	(2) {\em (Barycenters of ideal triangles) \cite[Chapter III.H, Lemmas 1.17, 3.3]{bridson-haefliger}} There are constants $D=D(\delta), R=R(\delta)$ such that the following holds.
	For any three distinct points $\xi_1,\xi_2, \xi_3\in \partial X$ there is a point $x\in X$ which is contained in the $D$-neighborhood
	of each of the three sides of any ideal triangle with vertices $\xi_1,\xi_2, \xi_3$. Moreover, 
	if $x'$ is any other such point then $d(x,x')\leq R$.  
\end{lemma}

A point $x$ as in Lemma \ref{lem-bary} 
is referred to as a {\em barycenter} of the ideal triangle.
Thus we have a coarsely well-defined map  $\br_X:\partial^{(3)} X\to X$ from the set of distinct triples of boundary points to $X$ sending a triple to its barycenter. We call such a map a {\em barycenter map.}
For $L\ge 0$, we say that the barycenter map is  {\em $L$-coarsely surjective} if the $L$-neighborhood of the image of any barycenter
map is all of $X$. Note that while a  barycenter map is only coarsely well-defined, $L$-coarse surjectivity is well-defined.

\begin{defn}\label{def:controlled fibers}
    The metric graph bundle $p:X\to B$   has \emph{controlled hyperbolic fibers} if there exist $\delta,L\ge 0$ such that for all $b\in \VV(B)$:
\begin{enumerate}
    \item $(F_b,d_b)$ is $\delta$-hyperbolic.
    \item The barycenter map $\br_{F_b}:\partial^{(3)}F_b\to F_b$ is $L$-coarsely surjective. 
\end{enumerate}
\end{defn}

\begin{defn}\label{def-qisxn}
	Suppose $p:X\to B$ is a metric graph bundle.
	Given $k\geq 1$ and a connected subgraph $\AAA\subset B$, a {\em $k$-quasi-isometric section} over $\AAA$ is
	a map $s:\AAA\to X$ such that $s$ is a $k$-quasi-isometric embedding and $p\circ s$ is the identity
	map on $\AAA$.\\
	A {\em quasi-isometric section} is a {\em $k$-quasi-isometric section} for some $k\geq 1$.
\end{defn}

For the group-theoretic version of Definition~\ref{def-qisxn} and Lemma~\ref{lem-qiexists}, see \cite{mosher-hypextns}. 
\begin{lemma} \cite[Proposition 2.10]{mahan-sardar}\label{lem-qiexists}
	If a metric graph bundle $p:X\to B$ has controlled hyperbolic fibers, then there exist $k\geq 1$ such that for each vertex $v$ of $X$, there  {is } a $k$-quasi-isometric section over $B$ containing $v$ in its image.
	In particular, if, in addition, $X$ is hyperbolic, then $B$ is hyperbolic.
\end{lemma}

The last statement follows from the fact that a quasi-isometrically embedded subset of a hyperbolic space is hyperbolic.

\subsection{The flaring condition}
\begin{defn}\label{defn-flare}\cite[Definition 1.12]{mahan-sardar}
	Suppose $p:X\rightarrow B$ is a metric  graph bundle.
	We say that it satisfies a {\em flaring condition} if for all   $k \geq 1$, there exist
	$\lambda_k>1$ and  $n_k,M_k\in \mathbb N$ such that
	the following holds:\\
	Let $\gamma:[-n_k,n_k]\rightarrow B$ be a geodesic and let
	$\tilde{\gamma_1}$ and $\tilde{\gamma_2}$ be two
	$k$-quasi-isometric sections of $\gamma$ in $X$.
	If $d_{\gamma(0)}(\tilde{\gamma_1}(0),\tilde{\gamma_2}(0))\geq M_k$,
	then we have
	\[\mbox{
		{\small $\lambda_k.d_{\gamma(0)}(\tilde{\gamma_1}(0),\tilde{\gamma_2}(0))\leq \mbox{max}\{d_{\gamma(n_k)}(\tilde{\gamma_1}(n_k),\tilde{\gamma_2}(n_k)),d_{\gamma(-n_k)}(\tilde{\gamma_1}(-n_k),\tilde{\gamma_2}(-n_k))\}$}}.
	\]
\end{defn}

Note that in Definition~\ref{defn-flare}, $\tilde{\gamma_i}$ is a $k-$quasi-isometric section in the sense of Definition~\ref{def-qisxn}. Thus 
$\tilde{\gamma_i}$ is really only defined on the vertex set (integers) of the domain 
$[-n_k,n_k]$ of $\gamma$.

\begin{theorem}\label{thm-flare-nec}\cite[Proposition 5.8 and Theorem 4.3]{mahan-sardar}
 Let $p: X \rightarrow B$ be a metric graph bundle with controlled hyperbolic fibers.
Then $X$ is hyperbolic if and only if the metric graph bundle satisfies a flaring condition and $B$ is hyperbolic.
\end{theorem}

\begin{lemma}\label{lem:general B to ray}
    Suppose that $p:X\to B$ is hyperbolic with controlled hyperbolic fibers, and $R\subseteq B$ is an infinite geodesic ray. Let $X_R = p^{-1}(R)$. Then the metric graph bundle $p:X_R \to R$ is hyperbolic with controlled hyperbolic fibers.
\end{lemma}

\begin{proof}
 By \cref{thm-flare-nec}, $p:X\to B$ satisfies the flaring condition. Hence, so does $p:X_R\to R$. Thus, by the same theorem, $p:X_R\to R$ (where $X_R$ is equipped with its intrinsic metric) is a hyperbolic metric graph bundle over a ray. 
\end{proof}

In view of \cref{lem:general B to ray}, we will mostly restrict our attention to hyperbolic bundles $p:X\to [0,\infty)$ where the ray $[0,\infty)$ is given its usual graph structure with vertices at integers.

\subsection{Cannon-Thurston maps
and dendrites}\label{sec-dendrite}
 \begin{defn}\label{def-ct}\cite{mitra-trees}
Let $(Y, d_Y),(X,d_X)$ be hyperbolic metric spaces and $i: Y \subset X$ be a proper embedding. Denote by $\widehat X = X \cup \partial X$ its boundary compactification. A \emph{Cannon-Thurston map} for the pair $(Y, X)$ is a map $\hat i :\hhat{Y}
\to \hhat X$ which is a continuous extension of $i$.
 \end{defn}

 \begin{theorem}\label{thm-ct}\cite[Theorem 5.3]{mahan-sardar}
 	Let $p: X \rightarrow [0,\infty)$ be a hyperbolic metric graph bundle with controlled hyperbolic fibers.
  Then the pair $(F_n, X)$ has a Cannon-Thurston map for each $n\in \bbN$.
 \end{theorem}

\begin{defn}\label{def-dend}
A dendrite $\DD$ is a compact Hausdorff space, so that for all $x\neq  y\in \DD $, there
is a unique unparametrized embedded arc $\sigma \subset \DD$ with end-points $x, y$.
\end{defn}

The following theorem due to Bowditch describes the Gromov boundary of $X$.
\begin{theorem}\cite[Theorem 1.2.5]{bowditch-stacks}\label{thm-bowditch}
 Let $p:X\to [0,\infty)$ be a hyperbolic metric graph bundle with controlled hyperbolic fibers. Then, the Gromov boundary $\partial X$ of $X$ is a dendrite.
\end{theorem}

\section{Polynomial Growth of Horospherical Subsets}
\subsection{Flowing quasiconvex subsets}\label{sec:flow}
Let $p:X\to[0,\infty)$ be a metric graph bundle with controlled hyperbolic fibers.  We pick $\delta$ sufficiently large such that all the $F_n$ are $\delta$-hyperbolic. 
Here, we assume that $[0,\infty)$ is given the standard simplicial structure with vertices at integers.
There exist (cf.\ \cite[Proposition 1.7]{mahan-sardar}) $K\geq 1, \ep \geq 0$ and quasi-isometries $\phi_n: F_{n-1} \to F_n$ such that 
\begin{enumerate}
\item $\phi_n$ is a $(K,\ep)$-quasi-isometry,
\item there is an edge in $X$ joining $x$ to $\phi_n(x)$ for all $x \in F_{n-1}$ and
all $n \in \natls$. 
\end{enumerate}

Let $\partial \phi_n: \partial F_{n-1} \to \partial F_n$ denote the 
homeomorphism induced by the quasi-isometry $\phi_n$. Let $\overline{\phi_n}$ be a coarse inverse of $\phi_n$, which we may assume, after increasing $K$ and $\ep$ if needed, is also a $(K,\ep)$-quasi-isometry. Let $\ZZ = \ZZ_0\subset  \partial F_{0}$ denote
a closed subset. Define $\ZZ_n$ inductively by $\ZZ_n = \partial \phi_n(\ZZ_{n-1})$.

If $C\subseteq X$ and $K\geq 1$, we define the \emph{$K$-thickening of $C$} to be the largest subgraph of $X$ contained in the metric $K$-neighborhood of  {$C$}. 

\begin{defn}\label{def-weakhull}
Let $(Y,d_Y)$ denote a $\delta$-hyperbolic metric space. For a closed subset $\UU \subset \partial Y$, the \emph{weak convex hull} $\whull(\UU)$ is the $4\delta$-thickening of the  union of all bi-infinite geodesics of the
form $(u,v)$, where $u\neq v \in \UU$.
\end{defn}

\begin{rmk}\label{rmk:hullqconvex}
		We observe that each $\whull(\UU)$ is a $4\delta$-quasiconvex connected subgraph of $Y$. This follows from the fact that geodesic quadrilaterals in
		{ $Y\cup \partial Y$} are $4\delta$-thin, and hence the union of all geodesics joining disjoint points of $\UU$ is $4\delta$-quasiconvex. 
\end{rmk}

Let $\CC_n =\whull(\ZZ_n)\subset F_n$ denote the weak convex hull of $\ZZ_n$ in the \emph{intrinsic}
hyperbolic metric on $F_n$. 
Since for each $n$, we have that $\partial \phi_{n}(\ZZ_{n-1})= \ZZ_n$ and $\CC_n$ is $4\delta$-quasiconvex in $F_n$,  there exists some $K'=K'(K,\epsilon,\delta)\geq 1$ such that $\dh(\phi_n (\CC_{n-1}),\CC_{n})\leq K'$ and $\dh(\overline{\phi_{n+1}} (\CC_{n+1}),\CC_{n})\leq K'$ for all $n$.
\begin{defn}\label{def-flowedimage}
The \emph{flowed image} $\FF\CC_0$ of $\CC_0$ in $X$ is defined to be the connected graph consisting of the subgraph 
$\bigcup_{n=0}^\infty \CC_n$, and all edges of the form $(u,v)$ where $u\in \CC_n$,  $v\in \CC_{n+1}$ and  either $d(\phi_n(u),v)\leq K'$ or $d(\overline{\phi_{n+1}}(v),u)\leq K'$.
\end{defn}

\begin{rmk}
		Although the flowed image $\FF\CC_0$ is not necessarily a subgraph of $X$ since the added edges need not lie in $X$,  its vertex set $\VV(\FF\CC_0)$ is a subset of $\VV(X)$.

	The same definition works  for a more general base $B$, which is not necessarily a ray. However, in view of the assumptions of \cref{lem:general B to ray}, we restrict our attention to bundles over rays for simplicity.
\end{rmk}

Let 
$ \pi_n : F_n \to \CC_n$ denote a nearest point projection of $F_n$ onto the (uniformly) quasiconvex set $\CC_n$ in the \emph{intrinsic hyperbolic} metric $d_n$ on $F_n$. We define $\Pi : X \to \FF\CC_0$ as follows. We define  $\Pi$ to be  $\bigcup_n\pi_n$ on $\bigcup_{n=0}^\infty F_n$, and then extend $\Pi$ on the remaining  edges  by sending the interior of an edge to the image of 
any one of its end-points. The following theorem, stating that $\Pi$ is a coarse Lipschitz retract is now a straightforward generalization of
\cite[Theorem 3.8]{mitra-trees} or \cite[Theorem 3.7]{mitra-ct}:

\begin{theorem}\label{thm-coarselipretract}
Let $p:X\to [0,\infty), \Pi$ be as above. There exists $C\geq 1$, such that for all $x, y \in X$, 
$$d(\Pi(x),\Pi(y)) \leq Cd(x, y) + C.$$
In particular,  $\VV(\FF\CC_0)$ is quasiconvex in $X$.
\end{theorem} 

\begin{proof}
	We sketch a proof for completeness. Let $x, y \in \VV(X)$. As in \cite[Theorem 3.8]{mitra-trees} or \cite[Theorem 3.7]{mitra-ct}, it suffices to show that there exists $C \geq 1$ such that $d(x,y) = 1$ implies 
$d(\Pi(x),\Pi(y)) \leq C$.\\

	\noindent {\bf Case 1: $x, y \in F_n$ for some $n$.}\\ In this case $\Pi(x)=\pi_n(x)$ and $\Pi(y)=\pi_n(y)$, where $\pi_n:X\to \CC_n$ is a nearest point projection onto a $4\delta$-quasiconvex subset. The statement now follows from \mbox{
\cite[{Chapter III.$\Gamma$, Proposition  3.11}]{bridson-haefliger}}\hskip0pt. \\

	\noindent {\bf Case 2: $x \in F_n, \,  y \in F_{n+1}$ for some $n$.}\\
	Since $\phi_{n+1}(x), y$  are at a uniformly bounded distance from each other, Case 1 allows us to assume that {$y=\phi_{n+1}(x)$ } without loss of generality. (This reduction is again exactly as in \cite[Theorem 3.8]{mitra-trees} or \cite[Theorem 3.7]{mitra-ct}.) Since $\CC_n$ is $4\delta$-quasiconvex in $F_n$, it follows that $[x,\pi_n(x)] \cup \CC_n$ is quasiconvex in $F_n$ with quasiconvexity constant dependent only on 
	$\delta$. Further, the Gromov inner product $\langle x, z \rangle_{\pi_n(x)}$ is bounded  above uniformly in terms of $\delta$ for any $z \in \CC_n$. Hence,
	$\langle \phi_{n+1}(x), \phi_{n+1}(z) \rangle_{\phi_{n+1}(\pi_n(x))}$   is 
	  bounded  above uniformly in terms of $\delta$ and the quasi-isometry constant of $\phi_{n+1}$, i.e.\ in terms of $\delta, K, \ep$.

	  Note also that  $\langle \phi_{n+1}(x), \phi_{n+1}(z) \rangle_{\pi_{n+1}(\phi_{n+1}(x))}$
	   is 
	  bounded  above uniformly in terms of $\delta$ (since $\phi_{n+1}(z)$ is contained in the $K'$-neighborhood of $\CC_{n+1}$).  It follows that {$d(\pi_{n+1}(\phi_{n+1}(x)),\phi_{n+1}(\pi_n(x)) )$ } is 
	  bounded  above uniformly in terms of  $\delta, K, \ep$. (See the last paragraph of the proof of  \cite[Lemma 3.5]{mitra-trees} for instance.)
	  This completes Case 2, and hence proves the inequality in the statement of the theorem.
	  
	  {The inequality in the theorem now gives one of the inequalities for $\FF\CC_0 \to X$ to be a quasi-isometric embedding.
	  	The other inequality is obvious from the construction of $\FF\CC_0$.
	  	Finally,  it follows from stability of quasi-geodesics in hyperbolic spaces that $\VV(\FF\CC_0)$ is quasiconvex in $X$.}
\end{proof}

\begin{rmk}\label{rmk-unifqc}
Note that the proof above shows that the constant $C$ depends only on  $\delta, K, \ep$. In particular, the quasiconvexity constant of  $\FF\CC_0$  depends  only on  $\delta, K, \ep$. 
\end{rmk}
For each flowed image $\FF\CC_0$, there is a surjective simplicial  map $p:\FF\CC_0\to [0,\infty)$ taking $\CC_n$ to $n$ and an edge of the form $(u,v)$ with $u\in \CC_n$ and $v\in \CC_{n+1}$ to $[n,n+1]$.
\begin{cor}\label{cor:flowedisMB}
	 Let $p:\FF\CC_0\to [0,\infty)$ be as above. Then $p$ is a  hyperbolic metric graph bundle. 
\end{cor}
\begin{proof}
	Since:
	\begin{enumerate}
		\item each $\CC_n$ is $4\delta$-quasiconvex in $F_n$ by Remark \ref{rmk:hullqconvex},
		\item there is a function $f$ such that each inclusion $F_n\to X$ is metrically proper as measured by $f$ by Definition \ref{defn-mgbdl},
		\item $\VV(\FF\CC_0)$ is quasiconvex in $X$ by Theorem \ref{thm-coarselipretract},
	\end{enumerate}
	it follows that there is a function $f':\bbN\to\bbN$ such that each inclusion $\CC_n\to \FF\CC_0$ is metrically proper as measured by $f'$. Remark \ref{rmk:hullqconvex} also says  each $\CC_n=p^{-1}(n)$ is a connected subgraph of $\FF\CC_0$,  so condition (1) of Definition \ref{defn-mgbdl} is satisfied. The choice of  $K'$ in Definition \ref{def-flowedimage} ensures that  condition (2) of Definition \ref{defn-mgbdl} is also satisfied, so $p$ is a metric graph bundle. Hyperbolicity of $\FF\CC_0$ follows from Theorem \ref{thm-coarselipretract}.
\end{proof}

We provide an alternate description of $\FF\CC_0$. 
For every pair $z_1 \neq  z_2 \in \ZZ$,  $\CC(z_1,z_2)$ is the $4\delta$-thickening of the  union of all bi-infinite geodesics joining
$z_1, z_2$ in $F_0$. Any two  such geodesics lie in
a $2\delta-$neighborhood of one  other. Denote the flowed image 
 $\FF\CC(z_1,z_2)$ of
 $\CC(z_1,z_2)$
by
$\LL(z_1,z_2)$. We refer to $\LL(z_1,z_2)$ as the \emph{ladder} corresponding
to $z_1, z_2$.

\begin{rmk}
In \cite{mitra-ct,mitra-trees}, the flowed image of a 
geodesic in $F_0$ was called a ladder. Thus, $\LL(z_1,z_2)$ is a $6\delta-$thickening of a ladder in the sense of \cite{mitra-ct,mitra-trees}.
Since these two notions are $6\delta-$coarsely equivalent, we continue to
use the terminology.
\end{rmk}

\begin{lemma}\label{lem-union} The vertex sets of $\FF \CC_0$ and  $\bigcup_{z_1\neq z_2 \in \ZZ} \LL(z_1,z_2)$ coincide.
 \end{lemma}

\begin{proof} Recall that $\FF\CC_0$ consists of $\bigcup_{n=0}^\infty \CC_n$ with some added edges. It is enough to show that   \[\bigcup_{n=0}^\infty \CC_n=\bigcup_{z_1\neq z_2 \in \ZZ} \LL(z_1,z_2).\]

The left-hand  side gives a union over $n=0, 1, 2, \cdots$ of the weak
convex hulls of $\ZZ_n$, where the latter is a union over all $z_1\neq z_2 \in \ZZ$. 

The right-hand  side is obtained by reversing the order in which this
	union is taken. For every $z_1\neq z_2 \in \ZZ$, the flowed image
		is $\LL(z_1,z_2)$. Taking the union now over  $z_1\neq z_2 \in \ZZ$ 
			gives the right-hand  side.
\end{proof}

Given a hyperbolic metric space $X$ and a subset $Y\subseteq X$, we let $\Lambda_X(Y)\subseteq \partial X$ denote the \emph{limit set} {of $Y$.

 \begin{lemma}\label{lem-flowhoroball}
Let  $p:X\to [0,\infty)$ be a hyperbolic  metric graph bundle with controlled hyperbolic fibers.
Let $\ZZ_n, \CC_n, \FF\CC_0, \LL(z_1,z_2)$ be as above. 
Suppose that the image of the Cannon-Thurston map $\hat{i}: \partial \CC_0 \to \partial X$ is a singleton $w$. Then there exists $C \geq 1$ such that
\begin{enumerate}
\item $ \LL(z_1,z_2)$ is $C-$quasiconvex in $\FF\CC_0$ for all $z_1 \neq z_2$,
\item {$\Lambda_X(\LL(z_1,z_2)) = \{w\}$}.
\end{enumerate}
\end{lemma}

\begin{proof} The first conclusion is an immediate consequence 
	of Theorem~\ref{thm-coarselipretract},  Remark~\ref{rmk-unifqc} and the fact that geodesics are convex.

The second conclusion follows from \cite[Section 4.2]{mitra-endlam} (see in particular,
Lemmas 4.8 and 4.9 of \cite{mitra-endlam}). We briefly sketch the essential features of the argument from  \cite[Section 4.2]{mitra-endlam} for completeness
as the present context is somewhat more general.

Let $l$ be a bi-infinite geodesic in $F_0$ joining $z_1,z_2$. 
{Let $\Pi_{12}$ denote the projection onto $ \LL(z_1,z_2)$
	constructed as in Theorem~\ref{thm-coarselipretract}.}
 For $o \in l$,
let $\sigma$ be a $K-$qi-section of $[0,\infty)$ in $X$ containing $o$. By Theorem~\ref{thm-coarselipretract}, {$\Pi_{12}\circ \sigma$} is a $CK-$qi-section of $[0,\infty)$ contained in $ \LL(z_1,z_2)$. 

Also, {$\Pi_{12}\circ \sigma ([0,\infty))$} coarsely separates $ \LL(z_1,z_2)$, i.e.\
$ \LL(z_1,z_2) \setminus \Pi_{12}\circ \sigma ([0,\infty))$ has precisely two unbounded components coarsely intersecting along r{blue}{$\Pi_{12}\circ \sigma ([0,\infty))$}. (This is
the content of \cite[Lemma 4.8]{mitra-endlam}, and the proof there goes through without change in the present context.)
Further, for $x_n \to z_1$ and $y_n \to z_2$ {with $x_n, y_n \in l$, } geodesics
$[x_n,y_n]_X$ in $X$ joining $x_n, y_n$ lie outside an $M_n-$ball about $o$, where
$M_n\to \infty$ as $n\to \infty$. This follows from the fact that the Cannon-Thurston map $\hat{i} : F_0 \cup \partial F_0 \to X \cup \partial X$ satisfies $\hat{i}(z_1) = \hat{i}(z_2)=w$. 

Let $\Pi_{12}([x_n,y_n]_X)$ denote the sequence of  images of vertices of  $[x_n,y_n]_X$ under $\Pi_{12}$, so that  $\Pi_{12}([x_n,y_n]_X)$
	is coarsely connected. By Theorem~\ref{thm-coarselipretract}, { $\LL(z_1,z_2) $ is $C-$quasiconvex, forcing $[x_n,y_n]_X$ to lie in a $C-$neighborhood of $\LL(z_1,z_2) $.
Therefore the inequality in Theorem~\ref{thm-coarselipretract} forces
$\Pi_{12}([x_n,y_n]_X)$ to lie in a uniformly bounded neighborhood of
$[x_n,y_n]_X$. }
	Hence, (after decreasing each $M_n$ by a uniformly bounded amount if necessary) 
 {$\Pi_{12}([x_n,y_n]_X)$ lies outside an $M_n-$ball about $o$, where
 	$M_n\to \infty$ as $n\to \infty$. }

 Since {$\Pi_{12}\circ \sigma ([0,\infty))$} coarsely separates $ \LL(z_1,z_2)$, and $x_n, y_n$ lie in separate components (for $n$ large enough), 
 {$\Pi_{12}\circ \sigma ([0,\infty))$ and $\Pi_{12}([x_n,y_n]_X)$} (coarsely) intersect at a point $w_n$ such that $d(w_n,o) \geq \frac{M_n}{C}$. Hence, $w_n \to w$ in $X \cup \partial X$. 
 Since $o$ was arbitrary, all qi-sections of $[0,\infty)$ contained in  $ \LL(z_1,z_2)$ converge to $w$. In particular, they are all asymptotic. (This completes a sketch of a proof of \cite[Lemma 4.9]{mitra-endlam} in the present
 context.)

 Finally, let 
 {
 	$\LL_n(z_1, z_2) = p^{-1}([n,\infty)) \cap \LL(z_1,z_2)$.}
 	{ Then, $d(o, \LL_n(z_1, z_2)) \geq n$. Further, $\LL_n(z_1,z_2)$ is the union over all $CK$-qi-sections of the form $\Pi_{12}\circ \sigma$ where $\sigma$ is a $K$-qi-section starting at a point in $\LL_n(z_1,z_2)$. 
 		As we have seen in the previous paragraph, these $CK$-qi-sections are all asymptotic to $w$. {Thus, $\LL_n(z_1,z_2)$ is the
 		union of uniform (independent of $n$) quasigeodesics, all asymptotic to $w$. Let $u_1, u_2
 		\in \LL_n(z_1,z_2)$. Then there exist uniform quasigeodesics $r_1, r_2$ containing $u_1, u_2$ respectively such that 
 		\begin{enumerate}
 				\item 	$r_1, r_2$ are asymptotic to $w$, and
 		\item 	$r_1, r_2$ are both
 		contained in $\LL_n(z_1,z_2)$.
 		\end{enumerate}
 	 Hence, the geodesic $[u_1, u_2]_X$ is contained in a uniformly bounded neighborhood of $r_1 \cup r_2$.
 		In particular,  $[u_1, u_2]_X$ is contained in a uniformly bounded neighborhood of $\LL_n(z_1,z_2)$. Since $u_1, u_2
 		\in \LL_n(z_1,z_2)$ are arbitrary, it}
 		follows that $\LL_n(z_1,z_2)$ is uniformly quasiconvex (independently of $n$). 
 	Thus, $\LL_n(z_1, z_2)$ lies in a small neighborhood  of $w$ with respect to the visual metric in $X \cup \partial X$. (In fact, the visual diameter of 
 $\LL_n(z_1, z_2)$ is of the order of	$\exp(-\lambda n)$, where $\lambda$ is independent of $n$.)}
 	Equivalently, $\LL_n(z_1, z_2) \to w$ as $n\to \infty$ in the Hausdorff topology on closed subsets of $X \cup \partial X$.
 {Let $q_n \in \LL(z_1,z_2)$ be a sequence which converges to a point in $\partial X$, then either $q_n$ has a subsequence which lies in a bounded neighborhood of $l$, in which case $q_n\to w$ since $\hat i(z_1)=\hat i(z_2)=w$ or it has a subsequence such that $q_n\in \LL_n(z_1,z_2)$ and so $q_n\to w$ as explained above. In either case $q_n \to w$ and we conclude that $\Lambda_X(\LL(z_1,z_2)) = \{w\}$.}
\end{proof}

\begin{lemma}\label{lem-lthoro1} We continue with the hypotheses of Lemma~\ref{lem-flowhoroball}.
Let $z_{1n} \neq  z_{2n} \in \ZZ$ such that $z_{1n} \to z_1, z_{2n} \to z_2$
and $z_1 \neq z_2$. Let 
{$q \in \Lambda_X (\FF\CC_0)$} be an accumulation point
of $ \LL(z_{1n},z_{2n})$. Then $q=w$.
\end{lemma}

 \begin{proof}
	{We first note the following. Let $K_n \subset X$ be a sequence of uniformly
		quasiconvex sets, all passing through a fixed $D_0-$neighborhood of $o \in X$, such that $K_n \to K_\infty$ in the Hausdorff topology on
		{closed subsets of} $X$. 	{(Here, as 
			below, the Hausdorff topology on 	closed subsets of $X$ is obtained as follows. First consider the Hausdorff topology on 	closed, and hence compact, subsets of the compact space $X\cup \partial X$. Then restrict this topology to closed subsets of
			$X$.  Note that the Hausdorff topology on compact subsets of $X\cup \partial X$ is obtained by first metrizing $X\cup \partial X$ and then equipping the collection of its compact subsets with the Hausdorff metric.)}
		
		Then any accumulation point in $\partial X$ of a sequence $\{z_n: z_n \in K_n\}$
	  lies in 
		$ \Lambda_X(K_\infty)$. Indeed, we can assume without loss of generality that
		$K_n$ consists of a union of geodesics from $o$, as the union of geodesics all meeting at $o$ is uniformly quasiconvex. Thus, we may assume that the sequence of geodesics $\{[o,z_n]: z_n \in K_n\}$
		has the property that $z_n \to z \in \partial X$. It follows that
		there exist points $y_n \in [o,z_n]$ such that
		\begin{enumerate}
		\item $d(y_n,o) \to \infty$; hence  $y_n \to z \in \partial X$.
		\item $d(y_n, K_\infty)$ is uniformly bounded.
		\end{enumerate}
	It follows that $z \in 	 \Lambda_X(K_\infty)$. }
		
			Since $z_{1n} \to z_1, z_{2n} \to z_2$ in $\partial \CC_0$, 
	the limit of $\CC(z_{1n},z_{2n})$  lies in a $2\delta-$neighborhood of $\CC(z_1,z_2)$. 
	Here convergence is taken in the Hausdorff topology on
		closed subsets of $\CC_0$, 	{i.e.\ we first consider the Hausdorff topology on 	closed subsets of the compact space $\CC_0\cup \partial \CC_0$	and restrict this to  	closed subsets of  $\CC_0$.}
		
	{	It follows that  any limit of $\LL(z_{1n},z_{2n})$  in the Hausdorff topology on
	closed subsets of $\FF\CC_0$) lies in a $2\delta-$neighborhood of $\LL(z_1,z_2)$. The Lemma now follows from the first paragraph of the proof using the fact that $\LL(z_{1n},z_{2n})$  is uniformly quasiconvex
	by Theorem~\ref{thm-coarselipretract}.}
\end{proof}

\begin{lemma}\label{lem-lthoro2} We continue with the hypotheses of Lemma~\ref{lem-flowhoroball}.
	Let $z_{1n} \neq  z_{2n} \in \ZZ$ such that $z_{1n} \to z, z_{2n} \to z$
	in $F_0 \cup \partial F_0$. Let $q \in \partial \FF\CC_0$ such that $q$ is an accumulation point
	of $ \LL(z_{1n},z_{2n})$. Then $q=w$.
\end{lemma}

\begin{proof}
Let $W_n=\{z_{1n},z_{2n}, z \} \subset \ZZ$. Then the visual diameter of $W_n$
with respect to the visual metric on $\partial F_0$ tends to zero as $n\to \infty$
(since  $z_{1n} \to z, z_{2n} \to z$). Let $\FF \whull (W_n)$ denote the flowed image of the weak convex hulls of $W_n$ as in Definition~\ref{def-flowedimage}. 

Let $o \in \CC_0 \subset \FF\CC_0$ denote a base-point, and $d$ denote
the metric on $ \FF\CC_0$.
The proof of \cite[Theorem 4.3]{mitra-ct} (see the last equation on \cite[p. 535]{mitra-ct}) now shows that   $d(\FF \whull (W_n),o) \to \infty$ 
as $n \to \infty$. We sketch a proof for completeness. It suffices to show that there exists  a proper function $M:\natls \to \natls$ such that $d_{\CC_0} ( \whull(\YY),o) \geq n$ implies that  $d_{\FF\CC_0} (\FF \whull (\YY),o) \geq M(n)$ for any closed subset $\YY \subset \partial \CC_0 \subset \partial F_0$. In fact, there is no reason to restrict to $\CC_0$ and it suffices to prove the existence of  a proper function $M:\natls \to \natls$ such that for any closed subset $\YY \subset \partial F_0$,
$d_{F_0} ( \whull(\YY),o) \geq n$ implies that  $d_{X} (\FF \whull (\YY),o) \geq M(n)$. 

 As before, let {$\YY_k \subset \partial F_k$ } denote the flowed image of $\YY=\YY_0$  in $ \partial F_k$, and let $\BB_k$ denote its weak hull in $F_k$.
Then there exists a constant $K_0$ depending only $K, \ep, \delta$ such that the Hausdorff distance between  $\BB_k$ and $\BB_{k+1}$ is at most $K_0$. {Let  $f:\natls \to \natls$ be such that the inclusion $F_0\to X$ } is proper as measured by $f$. Define $g:\natls \to \natls$ by $g(n)=\max f^{-1}([0,n-1])$. Observe that  if $d_{F_0}(x,y)\geq n$ for some $x,y\in F_0$ and $n\in \natls$, then $d_X(x,y)\geq g(n)$.  Hence if $d_{F_0}(\BB_0,o)\geq n$, it follows that $d_X(\BB_k,o)\geq g(n)-kK_0$. Since $d_X(\BB_k, o) \geq k$, 
 \[{g(n)\leq d_X(\BB_k,o)+kK_0\leq (K_0+1)d_X(\BB_k,o)}\] whenever $d_{F_0}(\BB_0,o)\geq n$.  Choosing $M(n) = 
\frac{g(n)}{K_0+1}$, we are done.

Next, note that $\FF \whull (W_n)$ is $C-$quasiconvex for some (uniform) $C \geq 0$ by Theorem~\ref{thm-coarselipretract}. 

Recall that
$\hat{i}$ denotes the Cannon-Thurston map (Theorem~\ref{thm-ct})
in this context. It follows that any accumulation point of $\FF \whull (W_n)$ in $\partial \FF\CC_0$ must equal $\hat{i}(z) = w$. Since 
$ \LL(z_{1n},z_{2n}) \subset \FF \whull (W_n)$, the conclusion follows. 
\end{proof}

\begin{cor}\label{cor-bdy1pt}
We continue with the hypotheses of Lemma~\ref{lem-flowhoroball}.
Then $\partial \FF \CC_0$ is
also a singleton.
\end{cor}

\begin{proof}
 Let $\ZZ=\partial \CC_0$,
and
 $\hat{i}(\ZZ)=w$. By Lemma~\ref{lem-union}, 
 $\FF \CC_0 = \bigcup_{z_1\neq z_2 \in \ZZ} \LL(z_1,z_2).$
By Lemma~\ref{lem-flowhoroball}, $\partial \LL(z_1,z_2) = w$ for all
$z_1\neq z_2$.
By Lemmas~\ref{lem-lthoro1} and \ref{lem-lthoro2}, any accumulation point of 
the collection $ \{\LL(z_1,z_2)\}$ contained in $\partial  F \CC_0$ also equals $w$. The corollary follows.
\end{proof}

\begin{rmk}\label{rmk-ctreg}
    A caveat is in order.  In general, if $p: X \to [0, \infty)$ is a
    hyperbolic metric graph bundle with uniformly, but not necessarily controlled,  hyperbolic fibers,
    the Cannon-Thurston map $\hat{i}: \partial F_0 \to \partial X$
 is not necessarily surjective. Let 
$ X$ be a half-space in the hyperbolic plane $\Hyp^2$, bounded by a 
bi-infinite geodesic $l$. Define $p: X \to [0,\infty)$
by setting $p(x)=d(x, l)$.
Then $p^{-1} (t)$
is the  set of points at distance $t $
 from $l$. Clearly, $p^{-1} (t)$ equipped with  its
intrinsic metric is isometric to the real line $\R$. Then 
 $p: X \to [0, \infty)$ can be easily discretized to give a
    hyperbolic metric graph bundle with {uniformly } hyperbolic fibers.

However, if $i: l(=F_0) \to X$ denotes the inclusion then the
Cannon-Thurston map $\hat{i}$
 is clearly not surjective in this case (in fact the 2 end-points of $l$ embed
into the semi-circular boundary of $X$). 
\end{rmk}

What Corollary~\ref{cor-bdy1pt}
shows is that unlike the situation in Remark~\ref{rmk-ctreg},
if the image of $\hat{i}$ is a singleton, then 
$\hat{i}$ is indeed a surjective map onto $\partial X$. In \cite[Proposition 2.6.1]{bowditch-stacks}, Bowditch shows that $\hat{i}$ is surjective if
the fibers
of  $p: X \to [0, \infty)$  are uniformly quasi-isometric to $\Hyp^2$.
We do not know how general a phenomenon this is.

\begin{qn}
Let $p: X \to [0, \infty)$ be a
    hyperbolic metric graph bundle with controlled hyperbolic fibers, such
    that fibers are uniformly quasi-isometric to a fixed {\bf one-ended}
    hyperbolic space $Y$. Is the Cannon-Thurston map $\hat{i}$ surjective?
\end{qn}

Remark~\ref{rmk-ctreg} shows that `controlled hyperbolic fibers'
cannot be replaced with `uniformly hyperbolic fibers' in the hypothesis.
\subsection{Growth of parabolic subsets}

\begin{lemma}\label{lem:polynomial growth of horospheres}
Let $p:X\to [0,\infty)$ be a hyperbolic metric graph bundle with controlled hyperbolic fibers and bounded valence. Let $\hat i:\widehat F_0 \to \widehat X$ be the Cannon-Thurston map.
If $\ZZ\subseteq \partial F_0$ is such that $\hat{i}(\ZZ)$ is a singleton, then $\CC_0=\whull(\ZZ)$ has polynomial growth with respect to the intrinsic metric $d_0$ of the fiber $F_0$.
\end{lemma}

\begin{proof}
    By \cref{thm-flare-nec}, the bundle satisfies a flaring condition. By \cref{lem-qiexists}, there exist $k\geq 1$ such that through every point
$u \in F_0$, there exists a $k-$quasi-isometric section $\sigma_u$  of the base $[0,\infty)$. If $u\in \CC_0$ then, by replacing $\sigma_u$ by $\Pi\circ \sigma_u$ if necessary, we may assume that $\sigma_u$ stays in the flowed image $\FF\CC_0$. 
By Corollary~\ref{cor-bdy1pt},  $\partial \FF\CC_0= \hat i(\ZZ)$ is a singleton.
By \cref{thm-coarselipretract} $\FF\CC_0$ is quasiconvex in $X$. 
Hence  all the quasi-isometric sections $\sigma_u$
are asymptotic. 

Let $u_0$ be a base-point in $\CC_0$. 
Let $d$ be the metric on $X$, and $d_n$ be the intrinsic metric in $\CC_n$.
The asymptotic $k$-quasigeodesic rays $\sigma_u, \sigma_v$ converge exponentially due to the flaring condition: There exist constants $M\ge 0 ,\lambda>1$ such that for all $u,v\in \CC_0$ $$ d_n(\sigma_u(n),\sigma_v(n))\le \frac{1}{\lambda^n} d_0(u,v)+M.$$ 
Hence, there exists a constant $k'\geq 1$ such that  $d_0-$ball of radius $r$ about $u_0$ in $\CC_0$ 
is contained in the $d-$ball of radius $R(r)=k'\log_\lambda (r)+k' $ about $u_0$ in $\CC_0$. 

If the valence of vertices in $X$ is bounded by $D$, then the $d$-ball of radius $R$ about any point has at most a constant times $D^R$ vertices.
We deduce that the $d_0$-ball of radius $r$ around $u_0$ in $\CC_0$ has at most a constant times $D^{R(r)}$ vertices. The proof is complete since $D^{R(r)} = D^{k'\log_\lambda (r)+k'}$ is polynomially growing in $r$.
\end{proof}

\section{Exponential Growth of Separators}\label{sec-qcsep}

Recall that if $X\to [0,\infty)$ is a hyperbolic metric graph bundle with controlled hyperbolic fibers, then $\partial X$ is a dendrite by \cref{thm-bowditch}. {Assume henceforth that $F_0$ is quasi-isometric to  a one-ended hyperbolic group $H$ so that $\partial F_0$ is connected. } Let $\hat i:\hat F_0\to \hat X$ be a Cannon-Thurston map. If $\hat{i}(\partial F_0)$ is a singleton, then
 $\partial X$ must be a singleton by Corollary~\ref{cor-bdy1pt} and hence by Lemma~\ref{lem:polynomial growth of horospheres} $F_0$ has polynomial growth, a contradiction.  Hence 
  $\hat{i}(\partial F_0) \subset \partial X$ is a connected subset of a dendrite. Since  $\hat{i}(\partial F_0)$ is not a singleton, it must itself be a dendrite.	 Let $\xi$ be a point that   separates $\hat{i}(\partial F_0)$.
   Then $\ZZ=\hat{i}^{-1}(\xi)$ separates $\partial F_0$.
In this section we will study the Hausdorff dimension  of such a separator $\ZZ$, and the growth of its  weak convex hull $\CC_0$. We shall denote Hausdorff
dimension by $\Hdim$.

Recall that if $H$ is a hyperbolic group then $\partial H$ has a visual metric $\rho$: i.e. $\rho$ induces the topology of $\partial H$ and there exists $a>1,C\ge 0$ such that if $\gamma$ is a bi-infinite geodesic connecting the points $\gamma(\pm\infty)\in \partial H$ then $$ a^{-d(e,\gamma) - C}\le \rho(\gamma(-\infty),\gamma(\infty))\le a^{-d(e,\gamma) + C}.$$

It is known that if $H$ is a one-ended hyperbolic group, then $\partial H$ is path connected and locally path connected. 
If $H$ does not virtually split over 2-ended subgroups, then by
a theorem of Bowditch \cite[Theorem 6.2]{bowditch-cutpts}, $\partial H$ does not have local cut points (cf. \cref{sec-jsj}).
In \cite[Theorem 3.3]{kleiner-icm}, Kleiner showed how one can upgrade the local connectedness of $\partial H$ to \emph{linear local connectedness} or \emph{LLC} for short.
The following lemma,  does the same for local cut points. We 
postpone its proof to the end of the paper, as its proof is somewhat technical.

For the metric space $(\partial H,\rho)$ we will use $B(\xi,r)$ and $B[\xi,r]$ to denote the open and closed $\rho$-balls of radius $r$ around $\xi$ respectively. For a subset $A\subset \partial H$, we denote by $\cl(A),\;\intr(A)$ and $\bdr(A)$ its closure, interior and boundary respectively.

\begin{namedtheorem}[\cref{lem: uniform no local cut points}]
	Let $H$ be a one-ended hyperbolic group such that $\partial H$ has no local cut points. Let $\partial H$ be equipped with a visual metric $\rho$. There exists $0<\lambda<1$ and $r_0>0$ so that for every $\xi\in \partial H$  and $0<r<r_0$ the following are satisfied:
	\begin{enumerate}
		\item  any two points in $B(\xi,\lambda r)$ are connected by a path in $B[\xi,r]$, and 
		\item  any two points in $\bdr(B[\xi,r])$ which are connected by a path in $B[\xi,r]$ can be connected by a path in $B[\xi,r]-B(\xi,\lambda r)$.
	\end{enumerate}
\end{namedtheorem}

We are now ready to state and prove the main proposition of this section:

\begin{prop}\label{prop: Hdim of qc separator}
	Let $H$ be a one-ended hyperbolic group such that $\partial H$ has no local cut points, and let $\rho$ be a visual metric on $\partial H$. Then there exists $\epsilon = \epsilon (\partial H,\rho)>0$ such that if $\ZZ\subseteq \partial H$ separates $\partial H$ then $\Hdim(\ZZ,\rho)>\epsilon.$ 
\end{prop}

\begin{proof}
    Denote $M=\partial H$ with the visual metric $\rho$.
	Assume that the points $\zeta_1,\zeta_2\in M - \ZZ$ belong to different components of $M-\ZZ$.
	\begin{claim}\label{claim-min}
		There exists a minimal closed subset (with respect to inclusion) $\ZZ'$ of $\ZZ$ that separates $\zeta_1,\zeta_2$.

		In particular, if $i\in \{1,2\}$ and  $[\zeta_i]$ denotes the path component of  $M-\ZZ'$ containing $\zeta_i$, then $\cl([\zeta_i]) \supseteq \ZZ'$.
	\end{claim}
	\begin{proof}
		By Zorn's Lemma it suffices to show that if $\{\ZZ_i\}_i$ is a totally ordered subset of the set of all closed subsets of $M$ that separate $\zeta_1,\zeta_2$, then so is $\ZZ_\infty=\bigcap_i \ZZ_i$. 
		Assume that $\zeta_1,\zeta_2$ are not separated by $\ZZ_\infty$ then (since $M$ is locally path connected, and $M - \ZZ_\infty$ is open) there exists a path $\gamma$ connecting $\zeta_1,\zeta_2$ in $M-\ZZ_\infty$. Since $\ZZ_i$ separates $\zeta_1,\zeta_2$ we must have $\gamma\cap \ZZ_i \ne \emptyset$ for all $i$. By the finite intersection property, $\gamma \cap \ZZ_\infty = \bigcap_i (\gamma \cap \ZZ_i) \ne \emptyset$ which gives a contradiction. This proves that a minimal separating subset exists.

		If $\ZZ'$ is a minimal separating subset then $\cl([\zeta_1])\supseteq \ZZ'$ as otherwise the subset $\ZZ''=\cl([\zeta_1])\cap \ZZ' \subsetneq \ZZ'$ will be a smaller separating set.
	\end{proof}

	From now on we will assume that $\ZZ$ is a closed minimal
		set as in Claim~\ref{claim-min}.
    Let $r_0$ be small enough so that $r_0 \le  \rho(\ZZ,\{\zeta_1,\zeta_2\})$  and for all $r<r_0$ \cref{lem: uniform no local cut points} holds. 
    Let $\lambda\in (0,1)$ be the constant from \cref{lem: uniform no local cut points}. 

	\begin{claim}\label{claim: non empty annulus}
		For all $\xi\in \ZZ$ and $0<r<r_0$ we have $\ZZ\cap (B[\xi,r]-B(\xi,\lambda r) ) \ne \emptyset$. 
	\end{claim}

	\begin{proof}
		For every $\xi\in \ZZ$, since $\ZZ$ is minimal we know that $\cl([\zeta_1]) \cap \ZZ = \cl([\zeta_2])\cap \ZZ = \ZZ$. 
		 We continue using the notation of  Claim~\ref{claim-min}.
		 	By minimality of $\ZZ$, $[\zeta_i]$ is the path component of 
		 	$(M - \ZZ)$ containing $\zeta_i$.
		So there are paths $\gamma_1, \gamma_2$ in $M-\ZZ$ connecting $\zeta_1,\zeta_2$ respectively to points in $B(\xi,\lambda r)$. By connecting the endpoints of $\gamma_1$ and $\gamma_2$ in $B[\xi,r]$ we can find a path $\gamma$ connecting $\zeta_1,\zeta_2$ so that $\gamma$ is in $B[\xi,r] \cup (M-\ZZ)$. 
		By compactness of $\gamma$, there are finitely many maximal subsegments of $\gamma$ that are contained in $B[\xi,r]$ and meet $B(\xi,\lambda r)$. By \cref{lem: uniform no local cut points} one can replace each one of them to obtain a path $\gamma'$ in $(B[\xi,r]-B(\xi,\lambda r)) \cup (M-\ZZ)$ connecting $\zeta_1,\zeta_2$. Such a path must intersect $\ZZ$, and so $\ZZ\cap (B[\xi,r]-B(\xi,\lambda r) ) \ne \emptyset$. 
	\end{proof}

	Set $0<\tau<1$ to be such that $K'=\tau\big(\lambda - 2\frac{\tau}{1-\tau}\big)>0$.
	Endow the Cantor set $\bfC=\{0,1\}^\bbN$ with the metric 
	$$d_{\bfC}((x_i),(y_i))=\inf\set{  \tau^n \;\middle|\; x_i=y_i \;\forall i\le n}.$$
    A simple computation gives $\Hdim(\bfC,d_\bfC)=-\log(2)/\log(\tau)$.

    To prove the proposition, we will construct a bi-Lipschitz embedding $f:(\bfC,d_{\bfC}) \to (\ZZ,\rho)$.
	It will follow that $\Hdim(\ZZ,\rho)\ge  \Hdim(\bfC,d_{\bfC})=-\frac{\log(2)}{\log(\tau)}$.

	First let us define $f$ for the set $\bfC_0$ of sequences which are eventually $0$, i.e. $\bfC_0 =\bigcup_N \{0,1\}^N\times \{0\}^\bbN \subset \bfC$. We do so by induction on $N$:
    For $N=0$, set $f((0)_n)=\xi_0$ for some arbitrary $\xi_0\in \ZZ$. For $N>0$ let $x=(x_n)$ be in $\{0,1\}^N\times \{0\}^\bbN$ and assume that $x_N=1$ (otherwise $f((x_n))$ was already defined by induction). Let $y_n$ be the sequence defined by $y_n = x_n$ if $n< N$ and $y_N=0$. Then, by induction $f((y_n))=\xi\in \ZZ$ is defined. By \cref{claim: non empty annulus} there exists some point $\xi'\in \ZZ \cap (B[\xi,r_N]-B(\xi,\lambda r_N))$ where $r_N=\tau^N r_0$, set $f((x_n))=\xi'$.

	Let us show that $f:\bfC_0\to \ZZ$ is $K$-bi-Lipschitz. If $x,y\in \bfC_0$ satisfy $d_\bfC(x,y)=\tau^{N}$ then $x$ and $y$ agree up to their $N$-th coordinate with a sequence $z\in \{0,1\}^N \times \{0\}^\bbN$. Recall that to construct $f(x)$ one starts with $f(z)$ and for every a change in the $i$-th coordinate moves to a point at distance at most $r_i$ from the previous one. By the triangle inequality we get $$ \rho(f(x),f(z))\le \sum _{i>N} r_i =\sum_{i> N}\tau^{i} r_0 = \frac{\tau r_0}{1-\tau}\tau^{N}.$$ The same holds for $f(y),f(z)$, and so by the triangle inequality
	$$\rho(f(x),f(y)) \le  2\frac{\tau r_0}{1-\tau}\tau^{N} =K'' \cdot d_\bfC(x,y)$$
	where $K''=2\frac{\tau r_0}{1-\tau}$.

 Similarly, 
	$$\rho(f(x),f(y)) \ge \lambda \tau^{N+1} r_0 - 2\sum_{i> N+1} \tau^i r_0 \ge \lambda r_0 \tau^{N+1} - 2 \frac{\tau r_0}{1-\tau} \tau ^{N+1} = K' \cdot d_\bfC(x,y),$$ 
 where $K'=\tau\big(\lambda r_0 -2\frac{\tau r_0}{1-\tau}\big)>0$. So $f$ is $K$-bi-Lipschitz for $K=\max\{K'',\frac{1}{K'}\}$.

	In particular, since $\ZZ$ is complete, we can extend $f:\overline{\bfC_0} = \bfC \to \ZZ$. This map remains $K$-bi-Lipschitz.
\end{proof}

\begin{lemma}\label{pos hdim implies exponential growth}
    If $H$ is a hyperbolic group and  $\ZZ \subseteq \partial H$ is a
    closed subset of positive Hausdorff dimension, then $\CC_0=\whull(\ZZ)$ is not polynomially growing.
\end{lemma}

\begin{proof}
    Consider the intersection $\CC_0\cap S(n)$ where $S(n)$ is the sphere of radius $n$ around $e$ in $H$.

    For every point $x\in H$, 
    consider the shadow $\Sh(e,x) \subseteq \partial H$ consisting of all the boundary points of geodesic rays starting at $e$ and passing through the ball of radius $4\delta$ around $x$ (see \cref{def: shadow}).
    By the $2\delta$ slimness of triangles (with ideal endpoints),  there exists $D=D(\delta)$ such that any two points in $\Sh(e,x)$ are connected by a geodesic whose distance from $e$ is at least $d(e,x)-D$.
    Since $\rho$ is a visual metric (with parameter $a$, and multiplicative constant $C$), this implies that $$\diam_\rho(\Sh(e,x))\le a^{-d(e,x)+D + C}.$$

    For large enough $n$, any geodesic ray $r$ that starts at $e$ and ends in $\xi\in \ZZ$, passes through the $4\delta$ neighborhood of some point $x\in \CC_0 \cap S(n)$. 
    It follows that the collection $\{\Sh(e,x)|x\in \CC_0 \cap S(n)\}$ covers $\ZZ$ for large enough $n$.
    For $0<\epsilon<\Hdim(\ZZ)$, consider the $\epsilon$-Hausdorff measure $\HH^\epsilon(\ZZ)$ of $\ZZ$. We have

    \begin{equation*}
    \infty=\HH^\epsilon(\ZZ) \le \liminf_{n\to\infty} \sum_{x\in \CC_0 \cap S(n)} \diam(\Sh(e,x))^\epsilon \le \liminf_{n\to\infty} |\CC_0\cap S(n)|\cdot a^{(-n+D +C)\epsilon}
    \end{equation*}

    And so $$\liminf_{n\to\infty} \frac{|\CC_0\cap S(n)|}{a^{n \epsilon }}=\infty $$
    It follows that $|\CC_0\cap S(n)|$ is not polynomially growing, and thus $\CC_0$ is not polynomially growing.
\end{proof}

In fact, it can be seen from the proof that $\CC_0$ has exponential growth. This however, will not be used in later proofs.

\section{Proof of Main Theorems}\label{sec:main theorems}

\begin{proof}[Proof of \cref{intro-main}]
Let $H$ be a one-ended hyperbolic group. Assume that $H$ is quasi-isometric to a fiber of a metric graph bundle $p:X\to B$ with controlled hyperbolic fibers, bounded valence and unbounded base. Then, the base $B$ has bounded valence, and so there exists an infinite geodesic ray in $B$. 
By \cref{lem:general B to ray} we may restrict the bundle to this ray to assume that $B=[0,\infty)$.
Recall that $\partial X$ is a dendrite by  \cref{thm-bowditch}.

By \cref{thm-ct}, there exists a Cannon-Thurston map $\hat i:\widehat F_0\to \widehat X$. 
Since $\partial H$ is connected, the image $\DD = \hat i(\widehat F_0)$ is a connected closed subset of $\partial X$, and so a dendrite.
If $\DD$ is a singleton, then by \cref{lem:polynomial growth of horospheres} $H$ is polynomially growing, in contradiction to it being a one-ended hyperbolic group. 
Otherwise, let $\xi \in \DD$ be a separating point in $\DD$.
Let $\ZZ=\hat i^{-1}(\xi)$, and let $\CC_0 = \whull(\ZZ)$.
By \cref{lem:polynomial growth of horospheres} $\CC_0$ is polynomially growing. 
Thus, by \cref{pos hdim implies exponential growth} $\ZZ$ cannot have positive Hausdorff dimension in $\partial F_0 =\partial H$. Since $\ZZ$ is a separating subset of $\partial H$, it follows by \cref{prop: Hdim of qc separator} that $\partial H$ has a local cut point. 
By 
a theorem of Bowditch \cite{bowditch-cutpts} (see \cref{thm: bowditch jsj} below), $H$ virtually splits over a 2-ended subgroup.
\end{proof}

\subsection{The end-decomposition and JSJ-decomposition.}\label{sec-jsj}
We provide here  a quick summary of the decomposition results we shall need about a hyperbolic group $H$, and their invariance under abstract commensurability.

\begin{defn}
An \emph{end-decomposition} for $H$ is a graph of groups decomposition of $H$ with  finite edge groups, and whose vertex groups are one-ended or finite.
\end{defn}

\begin{theorem}[{\cite{stallings1968torsion,dunwoody1985accessibility}}]\label{thm: end decomposition}
    Let $H$ be a finitely presented group. Then $H$ admits an end-decomposition. 
\end{theorem}

We mention here two important facts: the vertex groups of the end-decomposition are quasiconvex (with respect to any finite generating set), and the collection of all of them is commensurably-invariant in the following sense:

\begin{defn}
Let $H$ be a group. An \emph{abstract commensuration} of $H$ is an isomorphism $\phi:H_1 \to H_2$ between two finite index subgroups $H_1,H_2$ of $H$.

A collection $\KK$ of subgroups of a group $H$ is \emph{commensurably invariant} if for every $K\in \KK$ and every abstract commensuration $\phi:H_1\to H_2$ the image $\phi(K\cap H_1)$ is commensurable to some $K'\in \KK$.
\end{defn}

\begin{lemma}
    The collection of conjugates of vertex groups of an end-decomposition of a finitely presented group is commensurably invariant.
\end{lemma}

\begin{proof}
	Let $H$ be a finitely presented group and let $\KK$ be the collection of conjugates of one-ended vertex groups of end-decomposition of $H$. 
Since all finite subgroups of $H$ are commensurable,  it is enough to show  $\KK$ is commensurably invariant.  We note that $\KK$ consists precisely of the maximal one-ended subgroups of $H$.

Let $\phi:H_1\to H_2$ be an abstract commensuration of $H$.  If $K\in \KK$, then $\phi(K\cap H_1)$ is one-ended as it is isomorphic to a finite index subgroup of $K$. Hence $\phi(K\cap H_1)$ is contained in some $K'\in \KK$. Similarly, $\phi^{-1}(K'\cap H_2)$ is contained in some $K''\in \KK$. 
It follows that 
$$ K\cap H_1= \phi^{-1} (\phi(K\cap H_1) \cap H_2) \le \phi^{-1}(K'\cap H_2) \le K''.$$
As $K\cap H_1$ is a finite index subgroup of $K$, it follows that $K''=K$, and so $\phi(K\cap H_1)$ is commensurable to $K'$ as required.
\end{proof}

Next we move on to discuss
the JSJ-decompositions of hyperbolic groups following
Rips-Sela \cite{Sela,rips-sela,rips-sela-jsj} and Bowditch \cite{bowditch-cutpts}.

Following \cite{bowditch-cutpts}, a \emph{Fuchsian group} is a finitely generated group $L$ that acts properly discontinuously on the hyperbolic plane $\bbH^2$. Note that this allows for the action of $L$ to  not be faithful; however, the kernel must necessarily be finite.
The boundary subgroups of $L$ are the stabilizers of the elevations of the ends of $\bbH^2/L$.

\begin{defn}\label{def-jsj}
Let $H$ be a  one-ended hyperbolic group. A \emph{JSJ splitting} of $H$ is a finite graph of groups such that each edge group is 2-ended, and each vertex group is one of the following:
\begin{enumerate}
\item  a 2-ended subgroup,
\item  a maximal hanging Fuchsian subgroup, or
\item  a rigid subgroup.
\end{enumerate}
\end{defn}
Here,  we say that  a vertex group $L$ of a splitting of $H$ over 2-ended subgroups is a \emph{hanging Fuchsian subgroup} of $H$
if $L$ is Fuchsian
and the edge groups incident on $L$ are precisely the boundary subgroups of $L$.
 A vertex group $L$ of such a splitting is a \emph{maximal hanging Fuchsian  subgroup} if it is a hanging Fuchsian subgroup and there does not exist another splitting in which $L$ is properly contained in a hanging Fuchsian subgroup.

 Further, a \emph{rigid subgroup} $R$ of $H$
 is one which is
 \begin{enumerate}
 \item not 2-ended,
 \item elliptic in every splitting of $H$ over 2-ended subgroups. 
 \end{enumerate} 
 {If  $H$ is} virtually a \emph{surface group}, i.e. virtually the fundamental group  of a closed hyperbolic surface, then $H$ has a JSJ splitting consisting of a single maximal hanging Fuchsian vertex group and no edges.

In \cite[Theorem 0.1]{bowditch-cutpts}, Bowditch proves the existence of a canonical such JSJ splitting $T$ \emph{provided $H$ is not virtually a surface group}. In constructing the tree $T$,  Bowditch shows that there exists a collection $\Theta$ of cut pairs of $\partial H$ which:
\begin{enumerate}
	\item \label{item:cutpair_limit} consist precisely of the limit sets of edge stabilizers of $T$;
	\item \label{item:cutpt_topchar} is  invariant under homeomorphisms of $\partial H$.
\end{enumerate} 
In the terminology of \cite{bowditch-cutpts}, $\Theta$ consists of $\approx$-pairs, $\sim$-pairs, and jumps associated to infinite $\sim$-classes; this terminology is not needed in what follows.

Every commensuration $\phi:H_1\to H_2$ naturally induces  a self-quasi-isometry of $H$, and hence a boundary homeomorphism $\partial \phi:\partial H\to \partial H$. By (\ref{item:cutpt_topchar}), $\partial \phi$ preserves $\Theta$ and hence  $\phi(\Stab_H (\theta)\cap H_1)$ is a finite index subgroup of $\stab_H(\partial\phi(\theta))$ for all $\theta\in \Theta$. This shows that $\SSS(\Theta)\coloneqq \{\Stab_H(\theta)\mid \theta\in \Theta\}$ is commensurably invariant. By (\ref{item:cutpair_limit}),  $\SSS(\Theta)$ consists precisely of the maximal two-ended subgroups of $H$ containing some edge stabilizer of $T$ as a finite-index subgroup. This implies that the collection of edge stabilizers of $T$ is also commensurably invariant.

We summarize our discussion in the following theorem:

\begin{theorem}[{\cite{bowditch-cutpts}}]\label{thm: bowditch jsj}
    Let $H$ be a one-ended hyperbolic group, then $H$ has a unique JSJ splitting, and exactly one of the following holds:
    \begin{enumerate}
        \item $\partial H$ does not contain a local cut point. In this case,  $H$ does not virtually split over a 2-ended subgroup, and the JSJ splitting consists of a single vertex, which is rigid.
        \item $H$ is virtually a surface group. In this case,  the JSJ splitting consists of a single vertex, which is Fuchsian.
        \item The JSJ decomposition is non-trivial. In this case, the collection of edge stabilizers is commensurably invariant.
    \end{enumerate}
\end{theorem}

A comment is in order regarding a difference between Bowditch's notion of a canonical JSJ splitting that we use in this paper, and that of Rips-Sela's
(see \cite[p. 184]{bowditch-cutpts}).
Two-ended vertex groups (type 1 groups in Definition~\ref{def-jsj}) do not occur in \cite{Sela,rips-sela,rips-sela-jsj} forcing
edge groups incident on hanging Fuchsian   subgroups to be of finite index, and not equal as in \cite{bowditch-cutpts}.
 This forces the Rips-Sela splitting to be  canonical \emph{up to some
operations}. The canonical JSJ splitting of Bowditch that we use here is unique.

\subsection{Commensurated subgroups}\label{sec: commensurated subgroups}

\begin{lemma}[{\cite{GMRS}}]\label{lem: comm of quasiconvex}
    If $K$ is an infinite quasiconvex subgroup of a hyperbolic group $G$, then $[\Comm_G(K):K]<\infty$.
\end{lemma}

\begin{lemma}\label{comm-inv quasiconvex collection}
    Let $G$ be a hyperbolic group, and let $H$ be a commensurated hyperbolic subgroup of $G$ of infinite index.
    Then, $H$ does not contain a finite collection of subgroups $K_1,\dots,K_n$ such that
    \begin{enumerate}
        \item $K_1$ is quasiconvex in $G$, 
        \item $\KK=\{hK_ih^{-1}\;|\;h\in H\}$ is commensurably invariant in $H$.
    \end{enumerate}
\end{lemma}

\begin{proof}
    Assume that $\KK$ is commensurably invariant.
    Let $Hg_0,Hg_1,Hg_2,\dots$ be different cosets of $H$ in $G$.
    Conjugation by any element $g_n\in G$ restricts to  an abstract commensuration of $H$, 
    and so $g_n K_1 g_n^{-1}$ is commensurable to some $h_n K_{i_n} h^{-1}\in\KK$.
    By replacing $g_n$ by $h_n^{-1}g_n$, we may assume $g_n K_1 g_n^{-1}$ is commensurable to $K_{i_n}$.
    By passing to a subsequence we may assume that $i_n=i$ for all $i_n$.
    Setting $g_n'=g_0^{-1}g_n$, we see that $g_n'\in \Comm_G(K_1)$.
    The element $g_0$ commensurates $H$ and so there exist finite index subgroups $H',H''$ of $H$ so that $g_0 H' = H'' g_0$. 
    $H'' \le H$ and so $\{H''g_n\}_n$ are different cosets. Now
    $H'g_n' = H'g_0^{-1}g_n = g_0^{-1} H'' g_n$, and so 
    $\{H'g_n'\}_n$ are different cosets.
    Since $H'$ has finite index in $H$, we may pass to a subsequence and assume that $\{H g_n'\}_n$ are different cosets. 
    In particular, $\{K_1 g_n'\}_n$ are different cosets. We get $[\Comm_G(K_1):K_1]=\infty$ in contradiction to \cref{lem: comm of quasiconvex} and the assumption that $K_1$ is quasiconvex in $G$.
\end{proof}

Recall that any two-ended subgroup of a hyperbolic group is quasiconvex. Hence, as a corollary of \cref{thm: end decomposition,thm: bowditch jsj,comm-inv quasiconvex collection} we get:
\begin{cor}\label{JSJ of comm subgp is trivial}
    Let $G$ be a hyperbolic group, and let $H$ be an infinite commensurated hyperbolic subgroup of infinite index. Then, the one-ended vertex groups of the end-decomposition of $H$ have trivial JSJ-decomposition.
\end{cor}

\subsection{Cayley-Abels graphs and metric graph bundles}\label{sec-ca} Let $G$ be a finitely generated group with finite generating set $S$, and 
let $H<G$ be a commensurated subgroup, 
 i.e.\ $H^g \cap H$ is of finite index in $H$ for
all $g \in G$. 

\begin{defn}\label{def-ca}
	The \emph{Cayley-Abels graph} $\GG=\GG_{G/H}=\GG(G,H,S)$ of the pair $(G,H)$   with respect to $S$ is a simplicial graph with  vertex set $G/H$ and edge set $\{(gH,gsH)\mid g\in G,s\in S\}$.
\end{defn}
It is straightforward to see that $\GG$ is connected, locally finite and depends only on $(G,H)$ up to quasi-isometry; see for instance \cite[Proposition 2.17--2.19]{margolis2022discretisable}. Moreover, $G$ acts on $\GG_{G/H}$ by left-multiplication.  We note that if $H$ is normal in $G$, then $\GG_{G/H}$ is precisely the Cayley graph of the quotient $G/H$  (after removing loops and multiple edges). In particular,  $\GG(G,1,S)$ is  the Cayley graph of $G$ with respect to $S$.

\begin{rmk}
	For $G$, $H$ and $\GG_{G/H}$  as above, the closure of the image of the left action  $\rho: G\to \Aut(\GG_{G/H})$ is a compactly generated tdlc (totally disconnected locally compact) group $G\big/\!\big/H$ called the \emph{Schlichting completion} of $G$ with respect to $H$. We note that $\rho(G)$ is dense in $G\big/\!\big/H$ and the closure  $K$ of $\rho(H)$  is a compact open subgroup of $G\big/\!\big/H$.

	For a suitable choice of generating set, the Cayley--Abels graph of $G$ with respect to $H$ is precisely the Cayley--Abels graph of the tdlc $G\big/\!\big/H$, as constructed by Abels \cite{abels-cayley}.   We refer the reader to \cite{shalom-willis,bfmv} for further details on the Schlichting completion and the Cayley-Abels graph.
\end{rmk}
Let $G$ be a finitely generated group, $S$ a finite generating set and $H<G$ a finitely generated  commensurated subgroup. There is a naturally defined surjective simplicial \emph{quotient map} $p:\GG(G,1,S)\to \GG(G,H,S)$ which is given by $g\mapsto gH$ on vertices, and sends an edge $(g,gs)$ to  the edge $(gH,gsH)$ in the case $gH\neq gsH$, or to the vertex $gH$ in the case $gH=gsH$.

\begin{prop}[c.f.\ {\cite[Proposition 3.14]{margolis-gt}}]\label{prop-bdl}
	Let $G$ and $H$  be as above. There exists a finite generating set $S$ of $G$ such that the quotient map $p:\GG(G,1,S)\to \GG(G,H,S)$ is a metric graph bundle with fibers isometric to $\GG(H,1,S\cap H)$.
\end{prop}
\begin{proof}
	We pick a finite symmetric generating set $S_0$ of $G$ such that $S_0\cap H$ generates $H$. Since $H$ is commensurated in $G$, for each $s\in S_0$ there exists a finite set $F_s\subseteq H$ such that $H\subseteq F_s sHs^{-1}$. We set $F=\cup_{s\in S_0}F_s$, and may assume without loss of generality that $1\in F$. Set $S=FS_0$, which is a finite generating set of $G$ with the property that $S\cap H$ generates $H$.

	Let $\GG_G=\GG(G,1,S)$, $\GG_{G/H}=\GG(G,H,S)$ and $\GG_H=\GG(H,1,S\cap H)$, noting that $\GG_G$ and $\GG_H$ are simply the Cayley graphs of $G$ and $H$. Since $F_H=p^{-1}(H)=\GG_H$, there exists a function $f:\bbN\to\bbN$ such that the inclusion $\GG_H\hookrightarrow \GG_G$ is metrically proper as measured by $f$. Since $p$ is $G$-equivariant, the isometric action of  $G$ on $\GG_G$ is  transitive   on fibers of $p$, hence all fibers are isometric to $\GG_H$ and are  metrically proper as measured by $f$. 

	It remains to show that for each $g\in G$ and vertex $bH\in \GG_{G/H}$ adjacent to $gH=p(g)$, $g$ is adjacent to some vertex in $p^{-1}(bH)$. To see this, first note the definition of $\GG_{G/H}$ ensures that $b=gh_1fsh_2$ for some $h_1,h_2\in H$, $s\in S_0$ and $f\in F\subseteq H$. Since $h_1f\in H$, the definition of $F_{s}$ ensures that $h_1f=f'sh_3s^{-1}$ for some $f'\in F$ and $h_3\in H$. Since $f's\in S$, $gf's$ is adjacent to $g$ and \[p(gf's)=gf'sH=gf'sh_3h_2H=gf'sh_3s^{-1}sh_2H=gh_1fsh_2H=bH\] as required.
\end{proof}

The following is now an immediate consequence of Lemma~\ref{lem-qiexists}.

\begin{cor}\label{cor-qhyp}
	Let $G$ be a hyperbolic group with $H< G$ a hyperbolic commensurated subgroup.
Then the Cayley-Abels graph $\GG_{G/H}$ is also hyperbolic.
\end{cor}

Applying \cref{prop-bdl} and 
\cref{thm-flare-nec}, we deduce the following.
\begin{cor}\label{cor-flare}
		Let $G$ be a hyperbolic group with $H< G$ a hyperbolic commensurated subgroup. Then the metric graph bundle $p:\GG_G\to \GG_{G/H}$  as in Proposition~\ref{prop-bdl} satisfies a flaring condition.
\end{cor}

\subsection{Proof of \texorpdfstring{\cref{intro-hypcommens}}{Theorem \ref{intro-hypcommens}}}
\begin{proof} By \cref{thm: end decomposition} we can write $H$ as the fundamental group of a graph of groups $$\HH=(\Gamma,\{H_v\}_{v\in \VV(\Gamma)},\{H_e,i_e:H_e\to H_{t(e)}\}_{e\in \EE(\Gamma)})$$ whose vertex groups are one-ended or finite, and whose edge groups are finite.
By \cref{JSJ of comm subgp is trivial}, each one-ended $H_v$ has a trivial JSJ-decomposition. By \cref{thm: bowditch jsj}, either $\partial H_v$ has no local cut points or $H_v$ is virtually a surface group.

By \cref{prop-bdl,cor-flare}, we have a hyperbolic metric graph bundle $p:\GG_G\to \GG_{G/H}$ with controlled hyperbolic fibers. 
By \cref{lem:general B to ray}, we can restrict it to a hyperbolic bundle over a ray $p:X\to [0,\infty)$, such that each fiber $p^{-1}(n)=F_n$ is isometric to a Cayley graph of $H$.  We note that limit sets of conjugates of one-ended vertex groups of $\HH$ correspond precisely to infinite connected components of $\partial H$; see e.g.\ \cite[Proposition 3.1]{martinswi2015infended}.

For each one-ended vertex group $H_v$, let $\CC_0$ be the weak convex hull of the limit set $\Lambda_{F_0}( H_v)\subseteq \partial F_0$, which is at finite Hausdorff distance from $H_v$ since $H_v$ is quasiconvex in $F_0$ \mbox{
\cite{swenson2001quasiconvex}}\hskip0pt
.  The flowed image $\FF \CC_0$  is a  hyperbolic metric graph bundle $p:\FF \CC_0 \to [0,\infty)$  by Corollary \ref{cor:flowedisMB}. We claim $p:\FF \CC_0 \to [0,\infty)$ has controlled hyperbolic fibers. Indeed, each such fiber $\CC_n$  is the weak convex hull of an infinite connected component of $\partial F_n$. Since each $F_n$ is isometric to a Cayley graph of $H$,  there are only  finitely many isometry types of fibers, hence $\FF \CC_0$  has controlled hyperbolic fibers.
By \cref{intro-main}, $H_v$ virtually splits over a two-ended subgroup, and so $H_v$ is virtually a surface group.

We have shown that for every $v\in \VV(\Gamma)$, $H_v$ has a finite index subgroup $H'_v$ which is either trivial or a surface group. 
Let $N_v \lhd H_v$ be such that $N_v\le H'_v$ (in particular, $N_v$ is torsion free and so $N_v \cap i_e(H_e)=\{1\}$ for every edge $e$ such that $t(e)=v$). Let $F$ be the fundamental group of the graph of groups
$$(\Gamma,\{\overline H_v\}_{v\in \VV(\Gamma)},\{\overline H_e,\; i_e:\overline H_e\to \overline H_{t(e)} \}_{e\in \EE(\Gamma)}),$$
where $\overline H_v =H_v/N_v,\; \overline H_e =H_e$.
There is a surjection $\alpha:H\to F$.
The group $F$ is virtually free, and so it contains a finite index free subgroup $F'$ (in particular $F' \cap \overline H_v=\{1\})$. The preimage $\alpha^{-1}(F')$ is a finite index subgroup of $H$ which has a graph of groups whose edge groups are trivial, and whose vertex groups are surface groups or trivial.
\end{proof}

\begin{rmk}
 The conclusion of  \cref{intro-main} is weaker than that of \cref{intro-hypcommens}. However, for hyperbolic groups without 2-torsion, we
 can get a stronger conclusion.

 Let $H_0$ be a one-ended vertex group in the JSJ decomposition of $H$ (\cref{sec-jsj}).
	Note  that $H_0$ is quasiconvex in $H$, and choosing $\ZZ=\ZZ_0=\partial H_0$ in Definition~\ref{def-weakhull}, the flowed image $\FF\CC_0$ can be constructed as in
\cref{def-flowedimage}. Further, since weak hulls are uniformly quasiconvex, we can and shall regard (as in the proof of \cref{intro-hypcommens} above) $\FF\CC_0$ as a hyperbolic metric graph bundle with controlled hyperbolic fibers
with initial fiber quasi-isometric to $H_0$. Hence, by \cref{intro-main}, $H_0$ virtually splits over a 2-ended subgroup. In particular, it has a JSJ splitting
in the sense of \cref{sec-jsj}. Proceeding inductively, we obtain a JSJ hierarchy \cite{louder-touikan}.

Now, suppose in addition that $H$ has no 2-torsion. By the positive resolution of Swarup's conjecture on strong accessibility of such
hyperbolic groups \cite{louder-touikan}, the JSJ hierarchy terminates and
all one-ended groups at the bottom of the hierarchy are necessarily surface groups.

A theorem due to Carrasco-Mckay \cite{carrasco-mckay} then shows that
the conformal dimension of $\partial H$ is one. 
\end{rmk}

\subsection{Proof of \texorpdfstring{\cref{prop:conclusions}}{Proposition \ref{prop:conclusions}}}
\label{sec-prop-restrns}

If $\rho:G\to \Aut(\GG)$ is the action of $G$ on $\GG$, we recall that   the Schlichting Completion $G\big/\!\big/H$ is the closure of $\rho(G)$ in $\Aut(\GG)$; see Section \ref{sec-ca}. The Schlichting Completion is a totally disconnected locally compact group quasi-isometric to $\GG$ when equipped with the word metric with respect to a compact generating set; see for instance \mbox{
\cite[Example 4.C.4.4 \&  Theorem 4.C.5]{cornulierdelaharpe2016metric}}.\hskip0pt
 It  satisfies the property that the closure $K$ of $\rho(H)$ is  a compact open subgroup. We observe that if $G\big/\!\big/H$ is compact-by-discrete, then $K$ is commensurable to a compact open normal subgroup $K'\vartriangleleft G\big/\!\big/H$, whence $H$ is commensurable to the normal subgroup $\rho^{-1}(K')\vartriangleleft G$. We observe also that totally disconnected Lie groups are necessarily discrete. Thus if $G\big/\!\big/H$ is compact-by-Lie, then $H$ is commensurable to a normal subgroup of $G$. We use these observations to prove (1) and (2) of \cref{prop:conclusions}.

\begin{proof}[Proof of (1) and (2) of \cref{prop:conclusions}]
	By the preceding remarks, it is sufficient to show that the Schlichting completion is compact-by-Lie.	
	If $\GG$ is quasi-isometric to a rank-one symmetric space $X$, then so is $G\big/\!\big/H$. Quasi-isometric rigidity for such spaces \cite[Theorem 19.25]{cornulier2018quasiisometric}  implies that there is a compact  normal subgroup $K\vartriangleleft G\big/\!\big/H$ such that the quotient $Q=(G\big/\!\big/H)/K$ is a  closed subgroup of the  Lie group $\Isom(X)$. Thus $G\big/\!\big/H$ is compact-by-Lie, proving (1).

	The \emph{Hilbert--Smith Conjecture} says that if a locally compact group $G$ acts faithfully and continuously by homeomorphisms on a connected $n$-manifold, then $G$ is a Lie group. This is known to be true in dimensions $1\leq n\leq 3$ by work of Montgomery--Zippin and Pardon \cite{montgomeryzippin1955topological,pardon2013hilbertsmith}. In the case $\partial \GG$ is homeomorphic to $S^n$ for $1\leq n\leq 3$, the action of $G\big/\!\big/H$ on {$\partial\GG$ } has compact kernel, whose quotient is a Lie group by the Hilbert--Smith conjecture. 	
	In the case $\partial \GG=S^0$ (where $S^0$ is a pair of points), we have that $G\big/\!\big/H$ is quasi-isometric to $\bbR$. Thus \cite[Theorem 19.38]{cornulier2018quasiisometric}  implies $G\big/\!\big/H$ is compact-by-Lie. This proves $G\big/\!\big/H$ is compact-by-Lie in both cases,  proving (2).
	\end{proof}

	To prove (3) of \cref{prop:conclusions}, we use the characterization of amenable locally compact hyperbolic groups by Caprace--Cornulier--Monod--Tessera \cite{capracecornuliermonodtessera2015amenable}. In particular, we refer the reader  to  \cite{capracecornuliermonodtessera2015amenable} for the definitions of  \emph{compacting automorphisms} and \emph{vacuum sets}.

	\begin{proof}[Proof of (3) of \cref{prop:conclusions}]
	We argue by  contradiction: assume that  $G$ fixes a point of $\partial \GG$. Since $\rho(G)$ is dense in $G\big/\!\big/H$,  it follows that $G\big/\!\big/H$ also fixes a point of $\partial \GG$. As $|\partial \GG|>2$,  $G\big/\!\big/H$ is an amenable hyperbolic group  in the sense of \cite{capracecornuliermonodtessera2015amenable}. 

	Since $G\big/\!\big/H$ is totally disconnected and $\rho(G)$ is dense in $G\big/\!\big/H$,  the main theorem of \cite{capracecornuliermonodtessera2015amenable} implies that there exists some $g\in G$ such that $G\big/\!\big/H=L\rtimes \langle \rho(g)\rangle$, where $\rho(g)$ acts as a compacting automorphism of the totally disconnected noncompact group $L$. By \cite[Lemma 6.4]{capracecornuliermonodtessera2015amenable} there is some compact open subgroup $K\leq G\big/\!\big/H$ which is a vacuum set of $\rho(g)$ (see \cite{capracecornuliermonodtessera2015amenable} for the notion of  a vacuum set). Since all compact open subgroups are commensurable to one another, $\rho^{-1}(K)$ is commensurable to $H$. Replacing $H$ with $\rho^{-1}(K)$, we may assume that $\rho(H)$ is dense in $K$.

	The definition of a vacuum set ensures there exist $g_1,g_2\in G$ and $n$ sufficiently large such that $\rho(g_1)K\neq \rho(g_2)K$ and $\rho(g)^n\rho(g_i)K\rho(g)^{-n}\subseteq K$ for $i=1,2$. Thus  $g_iH\subseteq g^{-n}Hg^n$ for $i=1,2$. Since $g_1\in g_1H\subseteq g^{-n}Hg^n$, it follows that $H\leq g^{-n}Hg^n$. Since $H$ is commensurated, $H$ is a finite index subgroup of $g^{-n}Hg^n$. By Theorem \ref{intro-hypcommens}, $H$ is virtually a free product of hyperbolic surface and free groups, and is not two-ended by \mbox{
\cref{lem: comm of quasiconvex}}\hskip0pt
. Consequently, since free and hyperbolic surface groups have negative Euler characteristic, $H$   has negative rational Euler characteristic; see \mbox{
\cite[Proposition IX.7.3. b$'$, e \& e$'$]{brown1994cohomology}}\hskip0pt. In particular, as Euler characteristic is multiplicative with index,  $H$ cannot be isomorphic to a proper finite index subgroup of itself. This can also be deduced from a recent result of the first author \mbox{
\cite{lazarovich2023finite}}\hskip0pt.)  We conclude that $H=g^{-n}Hg^n$. Therefore $g_1H=g^{-n}Hg^n=g_2H$, contradicting the assumption $\rho(g_1)K\neq \rho(g_2)K$.
\end{proof}

\section{No local cut points}
The goal of this section is to prove the following technical lemma that was used in the proof
of \cref{prop: Hdim of qc separator}:

\begin{lemma}\label{lem: uniform no local cut points}
	Let $H$ be a one-ended hyperbolic group such that $\partial H$ has no local cut points. Let $\partial H$ be equipped with a visual metric $\rho$. There exists $0<\lambda<1$ and $r_0>0$ so that for every $\xi\in \partial H$  and $0<r<r_0$ the following are satisfied:
	\begin{enumerate}
		\item \label{local connectivity} any two points in $B(\xi,\lambda r)$ are connected by a path in $B[\xi,r]$, and 
		\item \label{no local cutpoint} any two points in $\bdr(B[\xi,r])$ which are connected by a path in $B[\xi,r]$ can be connected by a path in $B[\xi,r]-B(\xi,\lambda r)$.
	\end{enumerate}
\end{lemma}

Let $\calA$ be the set of pairs $(X_-,X_+)$ of non-empty open subsets $X_-,X_+ \subseteq \partial H$ such that $\cl(X_-)\cap \cl(X_+)=\emptyset$. 
    Consider the subset $\calA'\subseteq\calA$ of pairs $(X_-,X_+)\in \calA$ that satisfy the following two properties:
    \begin{enumerate}[label = ($\calA'_{\arabic*}$)]
        \item \label{p:lc} Any two points in $X_+$ can be connected by a path in $\partial H - X_-$.
        \item \label{p:lncp} Any two points in $\bdr (X_-)$ that can be connected by a path in $\partial H - X_-$ can be connected by a path in $\partial H - (X_-\cup X_+)$.
    \end{enumerate}

    \cref{lem: uniform no local cut points} is equivalent to showing that there exists $\lambda\in (0,1),r_0>0$ such that for all $0<r<r_0$ and $\xi\in \partial H$ we have $(\partial H - B[\xi,r],B(\xi,\lambda r))\in \calA'$. 

    \begin{lemma}\label{lem: non-uniform no local cut point}
        Let $H$ be as in \cref{lem: uniform no local cut points}. Then for every $\xi\in \partial H$ and small enough $\epsilon>0$ there exists $\delta>0$ such that $(\partial H - B[\xi,\epsilon],B(\xi,\delta))\in\calA'$.
    \end{lemma}
    \begin{proof}
        Since $H$ is one-ended, the boundary $\partial H$ is connected and locally path connected \cite{bowditch-cutpts,swarup-cutpoints}. 
        $\partial H$ does not contain local cut points 
        -- i.e. $\partial H - \xi$ is one-ended.
        The conclusion now follows using standard point-set topology.
    \end{proof}

    \begin{lemma}\label{claim: upward closed} If $(X_-,X_+)\in\calA'$ and $ (Y_-,Y_+)\in \calA$  are such that $X_\pm \supseteq Y_\pm$ then $(Y_-,Y_+)\in \calA'$
    \end{lemma}
    \begin{proof} It is clear that $(Y_-,Y_+)$ satisfies \ref{p:lc}. To show that $(Y_-,Y_+)$ satisfies \ref{p:lncp} consider two points $\xi_1,\xi_2$ on $\bdr(Y_-)$ that are connected by a path $\alpha$ in $\partial H - Y_-$.  The path $\alpha$ has finitely many maximal subpaths $\alpha_1,\dots,\alpha_n$ that are in $\partial H - X_-$ and intersect $X_+$. Note that the endpoints of $\alpha_i$ are in $\bdr(X_-)$. 
    Since $(X_-,X_+)$ satisfies \ref{p:lncp}, there exists paths $\alpha_i'$ in $\partial H - (X_-\cup X_+)$ connecting the same endpoints of $\alpha_i$.
    Replacing every $\alpha_i$ by $\alpha_i'$ in $\alpha$ gives rise to a path in $\partial H - (Y_-\cup Y_+)$ connecting $\xi_1,\xi_2$.
    \end{proof}

    The advantage of working with more general pairs $\calA$ is that it will give us the freedom of replacing balls in the visual metric by other sets, namely \emph{shadows}. We equip $H$ with a word-metric $d$ corresponding to a finite
    generating set. We suppose that $(H,d)$ is $\delta$-hyperbolic.

\begin{defn}\label{def: shadow}
    Let $\calS = \{(x_-,x_+)\in H \times H\;|\; d(x_-,x_+)>20\delta\}.$
    Define $\prec$ on $\calS$ by $(x_-,x_+) \prec (y_-,y_+)$ if $d(\{x_\pm\},\{y_\pm\})>20\delta$, $d(x_-,y_-)<d(x_+,y_-)$, and  a geodesic connecting $y_-,y_+$ passes through the balls of radius $2\delta$ around $x_-,x_+$.
    For $(x_-,x_+)\in \calS$ define the \emph{shadow of $x_+$ from $x_-$} to be the subset $\Sh(x_-,x_+)\subseteq \partial H$ of all endpoints of geodesic rays emanating from $x_-$ and passing through the ball of radius $4\delta$ around $x_+$.
\end{defn}

In the following lemma we collect a few useful facts about shadows, and relate shadows to balls in the visual metric.
\begin{lemma}\label{lem: facts about shadows}
    There exist $a>1$ and $C'\ge 0$ and $r_0>0$ such that for all $D\ge 20\delta$ there exists $\lambda\in (0,1)$ such that for every geodesic ray $\gamma$ emanating from $e$ and ending in some $\xi\in\partial H$:
    \begin{enumerate}[label=(\arabic*)]
        \item \label{f:opposite shadows} For all $(x_-,x_+)\in \calS$, $\cl(\Sh(x_-,x_+)) \cap \cl(\Sh(x_+,x_-))=\emptyset$.
        \item \label{f:shadows and prec}If $(x_-,x_+),(y_-,y_+)\in\calS$ and $(x_\pm)\prec (y_\pm)$ then $\Sh(y_-,y_+)\subseteq \Sh(x_-,x_+)$.
        \item \label{f:shadow basis}The collection 
        $\{\Sh(e,x_+)\}_{x_+ \in \gamma}$ 
        is a neighborhood basis for $\xi$.
       \item\label{f:balls and shadows} For all $x_+\in\gamma$ such that $(e,x_+)\in \calS$,
        \begin{equation*}
            B(\xi,a^{-d(e,x_+)-C'}) \subseteq \Sh(e,x_+) \subseteq B(\xi,a^{-d(e,x_+)+C'}).
        \end{equation*} 
    \item\label{f:distances and ratios} 
    for every $r<r_0$, there exist $x_-,x_+\in \gamma$ with $d(x_-,x_+)= D$ such that
        \begin{equation*}
    \partial H - B[\xi,r] \subseteq \intr(\Sh(x_+,x_-))
    \end{equation*}
    and
    \begin{equation*}
        B(\xi,\lambda r) \subseteq \intr(\Sh (x_-,x_+)).
    \end{equation*}
    \end{enumerate}
\end{lemma}

\begin{proof}
    \ref{f:opposite shadows} and \ref{f:shadows and prec} are clear. \ref{f:shadow basis} See \cite[Chapter III.H, Lemma 3.6]{bridson-haefliger}.
    \ref{f:balls and shadows} follows from the visual metric property.
    To prove \ref{f:distances and ratios}, set $\lambda = a^{-D-2C'-40\delta}$.
    Let 
    \begin{align*}
    x_-' &= \gamma(-\log_a(r)+C'),\\
    x_- &=\gamma(-\log_a(r)+C'+20\delta)\\
    x_+ &= \gamma(-\log_a(r)+C'+20\delta+D)\\
    x_+' &= \gamma(-\log_a(r)+C'+40\delta+D)
    \end{align*}
     By \ref{f:balls and shadows}, we have
     $$\intr\Sh(e,x_-') \subseteq B[\xi,r]$$
    It is easy to verify that $\intr(\Sh(e,x_-')) \cup \intr (\Sh(x_+,x_-)) = \partial H$
    and so 
    \begin{equation*}
    \partial H - B[\xi,r] \subseteq \intr(\Sh(x_+,x_-)).
    \end{equation*}

    By \ref{f:balls and shadows}, we have
    \begin{equation*} 
    B(\xi,\lambda r) \subseteq \intr(\Sh(e,x_+')). 
    \end{equation*}
    It is also easy to verify that 
    $$\intr(\Sh(e,x_+'))\subseteq \intr(\Sh (x_-,x_+)),$$
    and the second inclusion follows.
\end{proof}

In view of \cref{lem: facts about shadows}.\ref{f:distances and ratios} the ratios of radii of balls are related to distances between points in $H$.  

\begin{lemma}\label{lem: upgrading property of pairs}
Let $\calS'\subseteq \calS$. Assume that:
\begin{enumerate}
    \item \label{p: upward closed} (upward closed) if $(x_\pm)\in \calS'$ and $(x_\pm)\prec (y_\pm)$ then $(y_\pm)\in \calS'$,
    \item \label{p: H invariant} ($H$-invariant) if $(x_\pm)\in\calS'$ then $(gx_\pm)\in \calS'$ for all $g\in H$,
    \item \label{p: true for rays} (Eventual) and for every infinite geodesic ray $\gamma$ emanating from $x_-$ there exists $x_+$ on $\gamma$ such that $(x_-,x_+)\in \calS'$.
\end{enumerate}
Then, there exists $D\ge 20\delta$ such that for all $x_-,x_+\in H$ if $d(x_-,x_+)\ge D$ then $(x_-,x_+)\in \calS'$.
\end{lemma}

\begin{proof} We argue by 
 contradiction. Suppose that there is a sequence $(x^n_-,x^n_+) \notin \calS'$ such that $d(x^n_-,x^n_+) \to \infty$ .    
By \ref{p: H invariant}, we may assume that $x^n_- = e$ for all $n$.
we may also assume that the geodesics $\gamma_n=[e,x^n_+]$ converge to a geodesic ray $\gamma$ starting from $e$.
Pick a point $o_-$ on $\gamma$ at distance $20\delta$ from $e$. 
By \ref{p: true for rays}, there exists a point $o_+$ on $\gamma$ such that $(o_\pm)\in \calS'$ (and $d(e,o_+)>20\delta$).
Since $\gamma_n \to \gamma$, eventually $o_+$ lies on the geodesic $\gamma_n$ and $d(x_+,o_\pm)>20\delta$, and so $(o_\pm)\prec (x_\pm)$. By \ref{p: upward closed} $(x_\pm)\in \calS'$ in contradiction to the assumption.
\end{proof}

\begin{rmk}\label{rmk: rays at e} Note that by \ref{p: H invariant} it suffices to assume property \ref{p: true for rays} for rays emanating from $e$.
\end{rmk}

\begin{proof}[Proof of \cref{lem: uniform no local cut points}]

    Let us define $\calS'\subseteq\calS$ to be the subset of all $(x_-,x_+)\in \calS$ 
    such that $(X_-,X_+)\in\calA'$ where $X_\pm = \intr(\Sh(x_\mp,x_\pm))$.

    Let us verify that $\calS'$ satisfies the conditions of \cref{lem: upgrading property of pairs}:

    \textbf{Upward closed:} Let $(x_\pm) \prec (y_\pm)$.  Then it follows from the definition of $\prec$ that $(x_\mp )\prec (y_\mp)$. By \cref{lem: facts about shadows}.\ref{f:shadows and prec}, $X_\pm \supseteq Y_\pm$ where $X_\pm = \intr(\Sh(x_\mp,x_\pm))$ and $Y_\pm = \intr(\Sh(y_\mp,y_\pm))$. If $(x_\pm)\in \calS'$ then $(X_\pm)\in \calA'$. By \cref{claim: upward closed}, we get $(Y_\pm)\in \calA'$, and so $(y_\pm)\in\calS'$.

    \textbf{$H$-invariant:}  $H$ acts on $\partial H$ by homeomorphisms, and hence $\calA'$ is $H$-invariant. Clearly $g\Sh(x_\mp,x_\pm) = \Sh(gx_\mp,gx_\pm)$ for all $g\in H,(x_-,x_+)\in\calS$. Hence, $\calS'$ is $H$-invariant. 

    \textbf{Eventual:} By \cref{rmk: rays at e} it suffices to prove that $\calS'$ is eventual for rays emanating from $e$. Let $\gamma$ be a geodesic ray starting at $e$ and ending in $\xi\in \partial H$.
    By \ref{f:opposite shadows} and \ref{f:balls and shadows} of \cref{lem: facts about shadows}, and choosing  $\epsilon>0$ small enough we have that for all $x_+$ in $\gamma$ with $d(e,x_+)>20\delta$, 
    $\intr(\Sh(x_+,e)) \subseteq \partial H - B[\xi,\epsilon]$.
    By \cref{lem: non-uniform no local cut point} there exists $\delta>0$ such that $(\partial H - B[\xi,\epsilon],B(\xi,\delta))\in \calA'$. 
    By \cref{lem: facts about shadows}.\ref{f:shadow basis} there exists $x_+$ far enough in $\gamma$ so that $\intr(\Sh(e,x_+))\subseteq B(\xi,\delta)$.
    By \cref{claim: upward closed}, $(\intr(\Sh(x_+,e),\intr(\Sh(e,x_+))\in\calA'$, and so $(e,x_+)\in \calS'$.

    \bigskip

    We have verified that $\calS'$ satisfies the assumptions of \cref{lem: upgrading property of pairs}, and so there exists $D\ge 20\delta$ such that if $d(x_-,x_+)\ge D$ then $(x_-,x_+)\in \calS'$.

    Now, let $\lambda, r_0$ be as in \cref{lem: facts about shadows}.\ref{f:distances and ratios}, then for all $\xi\in\partial H$ and $0<r<r_0$ there exist $x_-,x_+$ on the geodesic $\gamma$ from $e$ to $\xi$ such that $d(x_-,x_+)=D$,  
    \begin{equation*}
    \partial H - B[\xi,r] \subseteq \intr(\Sh(x_+,x_-))\text{ and }B(\xi,\lambda r) \subseteq \intr(\Sh (x_-,x_+)).
    \end{equation*}
   We have $d(x_-,x_+)=D$ and so $(x_-,x_+)\in \calS'$, or equivalently $$(\intr(\Sh(x_+,x_-)),\intr(\Sh(x_-,x_+)))\in\calA'.$$
    By \cref{claim: upward closed}, the pair 
    $(\partial H - B[\xi, r], B(\xi,\lambda r))\in\calA'$. 
\end{proof}

\noindent {\bf Acknowledgments:} We thank Jingyin Huang for helpful comments on a previous draft. We are particularly thankful to an anonymous referee for carefully going through the paper and making many helpful comments.


\begin{thebibliography}{BFMvL23}

	\bibitem[Abe74]{abels-cayley}
	H.~Abels.
	\newblock Specker-{K}ompaktifizierungen von lokal kompakten topologischen
	{G}ruppen.
	\newblock {\em Math. Z.}, 135:325--361, 1973/74.

	\bibitem[Bes88]{bestvina-duke}
	Mladen Bestvina.
	\newblock Degenerations of the hyperbolic space.
	\newblock {\em Duke Math. J.}, 56(1):143--161, 1988.

	\bibitem[BF95]{BF-rtrees}
	Mladen Bestvina and Mark Feighn.
	\newblock Stable actions of groups on real trees.
	\newblock {\em Invent. Math.}, 121(2):287--321, 1995.

	\bibitem[BFMvL23]{bfmv}
	Nic Brody, David Fisher, Mahan Mj, and Wouter van Limbeek.
	\newblock Greenberg-shalom's commensurator arithmeticity hypothesis and
	applications.
	\newblock {\em preprint}, 2023.

	\bibitem[BH99]{bridson-haefliger}
	M.~Bridson and A~Haefliger.
	\newblock Metric spaces of nonpositive curvature.
	\newblock {\em Grundlehren der mathematischen Wissenchaften, Vol 319,
		Springer-Verlag}, 1999.

	\bibitem[BM91]{bes-mess}
	M.~Bestvina and G.~Mess.
	\newblock The boundary of negatively curved groups.
	\newblock {\em J.A.M.S. 4}, pages 469--481, 1991.

	\bibitem[Bow98]{bowditch-cutpts}
	Brian~H Bowditch.
	\newblock Cut points and canonical splittings of hyperbolic groups.
	\newblock {\em Acta Mathematica}, 180(2):145--186, 1998.

	\bibitem[Bow13]{bowditch-stacks}
	B.~H. Bowditch.
	\newblock Stacks of hyperbolic spaces and ends of 3-manifolds.
	\newblock In {\em Geometry and topology down under}, volume 597 of {\em
		Contemp. Math.}, pages 65--138. Amer. Math. Soc., Providence, RI, 2013.

\bibitem[Bro94]{brown1994cohomology}
{Kenneth~S. Brown.
	}\newblock {\em {Cohomology of groups}} {, volume~87 of }{\em {Graduate Texts in
		Mathematics}}.
	\newblock {Springer-Verlag, New York, 1994.
	}\newblock {Corrected reprint of the 1982 original.}

 \bibitem[CCMT15]{capracecornuliermonodtessera2015amenable}
	Pierre-Emmanuel Caprace, Yves Cornulier, Nicolas Monod, and Romain Tessera.
	\newblock Amenable hyperbolic groups.
	\newblock {\em J. Eur. Math. Soc. (JEMS)}, 17(11):2903--2947, 2015.

 \bibitem[CdlH16]{cornulierdelaharpe2016metric}
{Yves Cornulier and Pierre de~la Harpe.
	}\newblock {\em {Metric geometry of locally compact groups}}, volume~25 of 
	{\em {EMS Tracts in Mathematics}}.
	\newblock {European Mathematical Society (EMS), Z\"{u}rich, 2016.
	}\newblock {Winner of the 2016 EMS Monograph Award.
	}

 \bibitem[CM22]{carrasco-mckay}
	Matias Carrasco and John~M. Mackay.
	\newblock Conformal dimension of hyperbolic groups that split over elementary
	subgroups.
	\newblock {\em Invent. Math.}, 227(2):795--854, 2022.

	\bibitem[Cor18]{cornulier2018quasiisometric}
	Yves~de Cornulier.
	\newblock On the quasi-isometric classification of locally compact groups.
	\newblock In {\em New directions in locally compact groups}, volume 447 of {\em
		London Math. Soc. Lecture Note Ser.}, pages 275--342. Cambridge Univ. Press,
	Cambridge, 2018.

	\bibitem[Dun85]{dunwoody1985accessibility}
	Martin~J Dunwoody.
	\newblock The accessibility of finitely presented groups.
	\newblock {\em Inventiones mathematicae}, 81:449--457, 1985.

	\bibitem[GDLH90]{ghys1990espaces}
	{\'E}tienne Ghys and Pierre De~La~Harpe.
	\newblock Espaces m{\'e}triques hyperboliques.
	\newblock {\em Sur les groupes hyperboliques d’apr{\`e}s Mikhael Gromov},
	pages 27--45, 1990.

	\bibitem[GM22]{ghosh-mj}
	Pritam Ghosh and Mahan Mj.
	\newblock Regluing graphs of free groups.
	\newblock {\em Algebr. Geom. Topol.}, 22(4):1969--2006, 2022.

	\bibitem[GMRS98]{GMRS}
	R.~Gitik, M.~Mitra, E.~Rips, and M.~Sageev.
	\newblock Widths of {S}ubgroups.
	\newblock {\em Trans. Amer. Math. Soc. 350, no. 1}, pages 321--329, 1998.

	\bibitem[Gro85]{gromov-hypgps}
	M.~Gromov.
	\newblock Hyperbolic {G}roups.
	\newblock {\em in Essays in Group Theory, ed. Gersten, MSRI Publ.,vol.8,
		Springer Verlag}, pages 75--263, 1985.

	\bibitem[KK00]{kapovich-kleiner}
	Michael Kapovich and Bruce Kleiner.
	\newblock Hyperbolic groups with low-dimensional boundary.
	\newblock {\em Ann. Sci. \'{E}cole Norm. Sup. (4)}, 33(5):647--669, 2000.

	\bibitem[Kle06]{kleiner-icm}
	Bruce Kleiner.
	\newblock The asymptotic geometry of negatively curved spaces: uniformization,
	geometrization and rigidity.
	\newblock In {\em International {C}ongress of {M}athematicians. {V}ol. {II}},
	pages 743--768. Eur. Math. Soc., Z\"{u}rich, 2006.

 \bibitem[Laz23]{lazarovich2023finite}
{Nir Lazarovich.
	}\newblock {Finite index rigidity of hyperbolic groups.
	}\newblock {\em {arXiv preprint arXiv:2302.04484}}{, 2023.
	}

 \bibitem[LT17]{louder-touikan}
	Larsen Louder and Nicholas Touikan.
	\newblock Strong accessibility for finitely presented groups.
	\newblock {\em Geom. Topol.}, 21(3):1805--1835, 2017.

	\bibitem[Mar21]{margolis-gt}
	Alexander Margolis.
	\newblock The geometry of groups containing almost normal subgroups.
	\newblock {\em Geom. Topol.}, 25(5):2405--2468, 2021.

	\bibitem[Mar22]{margolis2022discretisable}
	Alex Margolis.
	\newblock Discretisable quasi-actions {I}: Topological completions and
	hyperbolicity.
	\newblock {\em arXiv preprint arXiv:2207.04401}, 2022.

	\bibitem[Min11]{min}
	Honglin Min.
	\newblock Hyperbolic graphs of surface groups.
	\newblock {\em Algebr. Geom. Topol.}, 11(1):449--476, 2011.

	\bibitem[Mit97]{mitra-endlam}
	M.~Mitra.
	\newblock Ending {L}aminations for {H}yperbolic {G}roup {E}xtensions.
	\newblock {\em Geom. Funct. Anal. vol.7 No. 2}, pages 379--402, 1997.

	\bibitem[Mit98a]{mitra-ct}
	M.~Mitra.
	\newblock Cannon-{T}hurston {M}aps for {H}yperbolic {G}roup {E}xtensions.
	\newblock {\em Topology 37}, pages 527--538, 1998.

	\bibitem[Mit98b]{mitra-trees}
	M.~Mitra.
	\newblock Cannon-{T}hurston {M}aps for {T}rees of {H}yperbolic {M}etric
	{S}paces.
	\newblock {\em Jour. Diff. Geom.48}, pages 135--164, 1998.

	\bibitem[Mos96]{mosher-hypextns}
	L.~Mosher.
	\newblock Hyperbolic {E}xtensions of {G}roups.
	\newblock {\em J. of Pure and Applied Algebra vol.110 No.3}, pages 305--314,
	1996.

	\bibitem[MS12]{mahan-sardar}
	M.~Mj and P.~Sardar.
	\newblock {A Combination Theorem for metric bundles}.
	\newblock {\em Geom. Funct. Anal. 22, no. 6}, pages 1636--1707, 2012.

	\bibitem[MS15]{martinswi2015infended}
	Alexandre Martin and Jacek \'{S}wiatkowski.
	\newblock Infinitely-ended hyperbolic groups with homeomorphic {G}romov
	boundaries.
	\newblock {\em J. Group Theory}, 18(2):273--289, 2015.

	\bibitem[MZ55]{montgomeryzippin1955topological}
	Deane Montgomery and Leo Zippin.
	\newblock {\em Topological transformation groups}.
	\newblock Interscience Publishers, New York-London, 1955.

	\bibitem[Par13]{pardon2013hilbertsmith}
	John Pardon.
	\newblock The {H}ilbert-{S}mith conjecture for three-manifolds.
	\newblock {\em J. Amer. Math. Soc.}, 26(3):879--899, 2013.

	\bibitem[Pau88]{paulin-inv}
	Fr\'{e}d\'{e}ric Paulin.
	\newblock Topologie de {G}romov \'{e}quivariante, structures hyperboliques et
	arbres r\'{e}els.
	\newblock {\em Invent. Math.}, 94(1):53--80, 1988.

	\bibitem[RS94]{rips-sela}
	E.~Rips and Z.~Sela.
	\newblock Structure and rigidity in hyperbolic groups.
	\newblock {\em Geom. Funct. Anal. v.4 no.3}, pages 337--371, 1994.

	\bibitem[RS97]{rips-sela-jsj}
	E.~Rips and Z.~Sela.
	\newblock Cyclic splittings of finitely presented groups and the canonical
	{JSJ} decomposition.
	\newblock {\em Ann. of Math. (2)}, 146(1):53--109, 1997.

	\bibitem[Sel97]{Sela}
	Z.~Sela.
	\newblock Structure and rigidity in ({G}romov) hyperbolic groups and discrete
	groups in rank {$1$} {L}ie groups. {II}.
	\newblock {\em Geom. Funct. Anal.}, 7(3):561--593, 1997.

	\bibitem[Sta68]{stallings1968torsion}
	John~R Stallings.
	\newblock On torsion-free groups with infinitely many ends.
	\newblock {\em Annals of Mathematics}, pages 312--334, 1968.

	\bibitem[SW13]{shalom-willis}
	Yehuda Shalom and George~A. Willis.
	\newblock Commensurated subgroups of arithmetic groups, totally disconnected
	groups and adelic rigidity.
	\newblock {\em Geom. Funct. Anal.}, 23(5):1631--1683, 2013.

	\bibitem[Swa96]{swarup-cutpoints}
	G.~A. Swarup.
	\newblock On the cut point conjecture.
	\newblock {\em Electron. Res. Announc. Amer. Math. Soc.}, 2(2):98--100, 1996.
	
	\bibitem[Swe01]{swenson2001quasiconvex}
{Eric~L. Swenson.
	}\newblock {Quasi-convex groups of isometries of negatively curved spaces.
	}\newblock {volume 110, pages 119--129. 2001.
	}\newblock {Geometric topology and geometric group theory (Milwaukee, WI, 1997).
	}

\end{thebibliography}
\end{document}